%% file: main.tex
\documentclass[10pt]{article}
\usepackage{graphicx}
\usepackage{subfigure}
\usepackage{float}
\usepackage{amsfonts,amsmath,amssymb,amsbsy,amsthm,enumerate}
\usepackage{bm}

\usepackage{tikz}
\usetikzlibrary{positioning,chains,fit,shapes,calc}
\usetikzlibrary{calc,positioning,decorations.pathmorphing,decorations.pathreplacing}

\usetikzlibrary{decorations.markings}
\usetikzlibrary{decorations.pathmorphing}
\usetikzlibrary{arrows.meta}
\usetikzlibrary{positioning,chains,fit,shapes,calc}
\usepackage{pdfpages}

\newcommand{\lct}{\mathrm{lct}}

\newcommand{\calV}{\mathcal{V}}
\newcommand{\calF}{\mathcal{F}}
\newcommand{\tw}{\mathrm{tw}}
\newcommand{\Branch}{\mathrm{Br}}
\newcommand{\cC}{{\mathscr C\!}}
\newcommand{\cB}{{\mathscr B\!}}
\newcommand{\twocross}{\mathcal{X}_2}
\newcommand{\A}{\mathcal{A}}

\usepackage{mathrsfs} 
\usepackage{comment}

\showboxdepth=\maxdimen
\showboxbreadth=\maxdimen

\newtheorem{theorem}{Theorem}[section]

\newtheorem{conjecture}[theorem]{Conjecture}
\newtheorem{proposition}[theorem]{Proposition}
\newtheorem{lemma}[theorem]{Lemma}
\newtheorem{corollary}[theorem]{Corollary}

\newtheorem{claim}{Claim}

\definecolor{myblue}{RGB}{80,80,160}
\definecolor{mygreen}{RGB}{80,160,80}
\definecolor{myred}{RGB}{255,0,0}
\definecolor{mybrown}{RGB}{139,69,19}

\newtheoremstyle{case}{}{}{}{}{\bfseries}{:}{ }{}
\theoremstyle{case}
\newtheorem{case}{Case}

\numberwithin{subcase}{case}

\usepackage{marginnote}

\title{All longest cycles intersect in partial 3-trees}
\author{Juan Guti\'errez\textsuperscript{1}\\
	{\footnotesize \textsuperscript{1}Departamento de Ciencia de la Computación}\vspace{-.2cm}\\
	{\footnotesize Universidade de Ingeniería y Tecnología, Perú}
	\footnote{A previous version of this paper was supported by FAPESP (Proc.~2015/08538-5).
  E-mail:
  jgutierreza@utec.edu.pe.}
}

\begin{document}

\maketitle

\begin{abstract}
  We show that all longest cycles intersect in 2-connected
  partial 3-trees.
\end{abstract}

\section{Introduction}

\input{sectionIntroduction.tex}

\section{Paths and Cycles}\label{section:basic-concepts-pathscycles}

\input{sectionPathsCyclesAttractor.tex}

\section{Tree Decomposition and Branches}
\label{section:basic-concepts-treewidth}

\input{sectionTreeDecAndBranches.tex}

\section{Proof of the main theorem}
\label{section:maintheorem}
\input{sectionMainTheorem.tex}

\section{Proof of the main lemma}
\label{section:mainlemma}
\input{sectionMainLemma.tex}

\section{Proof of the auxiliary lemmas}
\label{section:auxiliarylemmas}
\input{sectionAuxiliaryLemmas.tex}

\section{Concluding remarks}
 \label{section:conclusion}

\input{sectionConcludingRemarks}

\bibliography{bibliografia}

\end{document}

%% file: sectionIntroduction.tex
It is known that, in a 2-connected graph, 
every pair of longest cycles intersect each other in at least two vertices.
A natural question is whether all longest
cycles have a vertex in common in 2-connected graphs (if the graph is not 2-connected, two longest cycles can be disjoint).
This has in general a negative answer, 
as the Petersen's graph shows.
However, there are some graph classes for which this question has a positive answer, such
as
dually chordal graphs \cite{Jobson16} (a class of graphs
that includes doubly chordal, strongly chordal, and interval graphs)
and 3-trees \cite{Fernandes17}.
In this paper, we generalize the later result
by showing that all longest cycles intersect in 2-connected
graphs with treewidth at most~3,
also known as partial~3-trees.
A previous extended abstract containing this result was presented
at LATIN 2018 \cite{Gutierrez18}.
A similar question for paths instead of cycles has been approached in \cite{Cerioli2020}.

This paper is organized as follows.
In Section~\ref{section:basic-concepts-pathscycles}, we establish basic concepts on paths and cycles.
In Section~\ref{section:basic-concepts-treewidth},
we give definitions on tree decompositions
and branches.
In Section~\ref{section:maintheorem},
we state
the main lemma (Lemma~\ref{lemma:lct-tw3-lct=1-or-fenced})
and proceed to the proof of the main result (Theorem~\ref{theorem:lct-tw3-lct=1}).
The other two sections contain the most technical
parts of the proof.
Section~\ref{section:mainlemma} contains the proof of the main lemma,
using auxiliary results proved in the Section \ref{section:auxiliarylemmas}.
Finally, in Section~\ref{section:conclusion} we present some concluding remarks.
In this paper, all graphs considered are simple and the notation used is standard \cite{BondyM08,Diestel10}.



%% file: sectionPathsCyclesAttractor.tex
Given paths~$C'$ and~$C''$,
if~$(V(C') \cup V(C''), E(C') \cup E(C''))$ is a path or a cycle,
it is denoted by~$C' \cdot C''$.
For a vertex~$v$ in a path~$P$, let~$P'$ and~$P''$ be
the paths such that~$P = P' \cdot P''$ with~$V(P') \cap V(P'')=\{v\}$.  
We refer to these two paths as the  \emph{$v$-tails} of~$P$.
For a pair of vertices~$\{a,b\}$ in a cycle~$C$, 
let~$C'$ and~$C''$ be the paths 
such that~$C = C' \cdot C''$ and~$V(C') \cap V(C'') = \{a,b\}$.  
We refer to these paths as the \emph{$ab$-parts}\index{parts of a cycle} of~$C$.
Moreover, we can extend this notation and define,
for a triple of vertices~$\{a,b,c\}$
in a cycle~$C$, the \emph{$abc$-parts} of~$C$;
and, when the context is clear, we denote by $C_{ab}$, $C_{bc}$, and $C_{ac}$ the corresponding~$abc$-parts of~$C$.
A similar notation is used to define the
the \emph{$abcd$-parts} of~$C$ for a given subset of four vertices~$\{a,b,c,d\}$ in $C$.

Let~$G$ be a graph and let~$S \subseteq V(G)$. 
We say that~$S$ \emph{separates} vertices~$u$ and~$v$ if~$u$ and~$v$ are
in different components of~$G-S$.
Let~$X \subseteq V(G)$. We say that~$S$ separates~$X$
if~$S$ separates at least two vertices of~$X$.

We say that a path or cycle~$C'$ \emph{$k$-intersects}~$S$\index{$k$-intersects}
if~${|V(C') \cap S|=k}$.
Moreover, we also say that~$C'$~$k$-intersects~$S$ at~$V(C') \cap S$.
A path or cycle~$C'$ \emph{crosses}~$S$ if~$S$ separates~$V(C')$
in~$G$. Otherwise,~$S$ \emph{fences}~$C'$.\index{fence} 
If~$C'$ crosses~$S$ and~$k$-intersects~$S$, then we say 
that~$C'$~$k$-\emph{crosses}~$S$.\index{$k$-cross}
We also say that~$C'$~$k$-crosses~$S$ at~$V(C') \cap S$.
If~$C'$ is fenced by~$S$ and~$k$-intersects~$S$, then we say that~$C'$ 
is~$k$-\emph{fenced} by~$S$.
Two cycles are $S$-equivalent
if they intersect $S$ at the same set of vertices. 
(Figure \ref{fig:defs-fenced-crosses}).
\footnote{The terms fenced and crossing were first coined at \cite{DeRezende13}.}

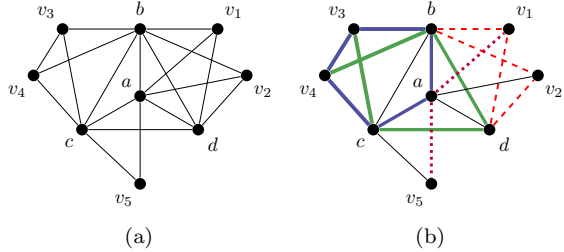
\begin{figure}[h]
\centering
\subfigure[ ]{
\resizebox{.3\textwidth}{!}{\input{Figures/defs-fenced-crosses-1}}
}
\subfigure[ ]{
\resizebox{.3\textwidth}{!}{\input{Figures/defs-fenced-crosses-2}}
}
\caption{(a) A graph~$G$ with~$S=\{a,b,c,d\}$.
    (b) Consider paths~$P_1=v_1av_5$ and~$P_2=v_3cdbv_4$,
    and cycles~$C_1=v_1bv_2dv_1$ and~$C_2=v_3v_4cabv_3$. Then~$P_1$ and~$C_1$
    cross~$S$, and~$P_2$ and~$C_2$ are fenced by~$S$.
    Moreover,~$P_1$ 1-crosses~$S$,~$P_2$ is~$3$-fenced by~$S$,
    $C_1$ 2-crosses~$S$ and~$C_2$ is 3-fenced by~$S$.
    (Also note that path~$cd$ and cycle~$abda$ are fenced by~$S$.)
    Paths~$P_2$ and~$v_1bcdv_2$ are~$S$-equivalent.
    Cycles~$C_2$ and~$v_1bcv_5av_1$ are~$S$-equivalent.
    }
\label{fig:defs-fenced-crosses}
\end{figure}

A cycle in a graph~$G$ is called a \emph{longest cycle}
if it has maximum length over all cycles in~$G$.
Let~$S \subseteq V(G)$.
The next proposition is well known.
We use it
several times through the text without making any reference
to it.

\begin{proposition}\label{prop:two-longest-cycles-intersect}
Let~$C$ and~$D$ be a pair of longest cycles in a 2-connected graph~$G$.
Then~${|V(C) \cap V(D)| \geq 2}$.
\end{proposition}

%% file: Figures/defs-fenced-crosses-1.tex
\begin{tikzpicture}
[scale=0.7, label distance=3pt, every node/.style={fill,circle,inner sep=0pt,minimum size=6pt}]
 
    \node at (0,0)[myblue,label=below left :$c$,fill=black, circle](c) {};
    \node at (3,0)[myblue,label=below right:$d$,fill=black, circle](d) {};
    \node at (1.5,2.5981)[myblue,label=above :$b$,fill=black, circle](b) {};
    \node at (3.5,2.5981)[myblue,label=above right:$v_1$,fill=black, circle](v1) {};
    \node at (4.25,1.4)[myblue,label=below right:$v_2$,fill=black, circle](v2) {};
    \node at (-0.5,2.5981)[myblue,label=above left:$v_3$,fill=black, circle](v3) {};

    \node at (1.5,0.866)[myblue,label=above   left:$a$,fill=black, circle](a) {};
    \node at (-1.25,1.4)[myblue,label=below left:$v_4$,fill=black, circle](v4) {};
    \node at (1.5,-1.4)[myblue,label=below left:$v_5$,fill=black, circle](v5) {};

    \draw  (a) -- (b);
    \draw  (b) -- (c);
    \draw  (a) -- (c);
    \draw (a) -- (d);
    \draw (b) -- (d);
    \draw (c) -- (d);
    
    \draw (b) -- (v1);
    \draw (d) -- (v1);
    \draw (b) -- (v2);
    \draw (d) -- (v2);
    
    \draw (b) -- (v3);
    \draw (c) -- (v4);
    \draw (v3) -- (v4);
    \draw (v3) -- (c);
    \draw (v4) -- (b);
    
    \draw (a) -- (v1);
    \draw (a) -- (v2);
    \draw (a) -- (v5);
    \draw (v5) -- (c);

\end{tikzpicture}

%% file: Figures/defs-fenced-crosses-2.tex
\begin{tikzpicture}
[scale=0.7, label distance=3pt, every node/.style={fill,circle,inner sep=0pt,minimum size=6pt}]
 
    \node at (0,0)[myblue,label=below left :$c$,fill=black, circle](c) {};
    \node at (3,0)[myblue,label=below right:$d$,fill=black, circle](d) {};
    \node at (1.5,2.5981)[myblue,label=above :$b$,fill=black, circle](b) {};
    \node at (3.5,2.5981)[myblue,label=above right:$v_1$,fill=black, circle](v1) {};
    \node at (4.25,1.4)[myblue,label=below right:$v_2$,fill=black, circle](v2) {};
    \node at (-0.5,2.5981)[myblue,label=above left:$v_3$,fill=black, circle](v3) {};

    \node at (1.5,0.866)[myblue,label=above   left:$a$,fill=black, circle](a) {};
    \node at (-1.25,1.4)[myblue,label=below left:$v_4$,fill=black, circle](v4) {};
    \node at (1.5,-1.4)[myblue,label=below left:$v_5$,fill=black, circle](v5) {};

    \draw  [color=myblue][line width=1.7pt] (a) -- (b);
    \draw  (b) -- (c);
    \draw  [color=myblue][line width=1.7pt] (a) -- (c);
    \draw (a) -- (d);
    \draw  [color=mygreen] [line width=1.7pt] (b) -- (d);
    \draw  [color=mygreen] [line width=1.7pt] (c) -- (d);
    
    \draw [dashed] [color=red] [line width=1.0pt] (b) -- (v1);
    \draw [dashed] [color=red] [line width=1.0pt] (d) -- (v1);
    \draw [dashed] [color=red] [line width=1.0pt] (b) -- (v2);
    \draw [dashed] [color=red] [line width=1.0pt] (d) -- (v2);
    
    \draw  (a) -- (v2);

    \draw [color=myblue][line width=2pt](b) -- (v3);
    \draw  [color=myblue][line width=2pt](c) -- (v4);
    \draw  [color=myblue][line width=2pt](v3) -- (v4);
    \draw  [color=mygreen] [line width=2pt] (v3) -- (c);
    \draw  [color=mygreen] [line width=2pt] (v4) -- (b);
    
    \draw  [dotted] [color=purple][line width=1.5pt](a) -- (v1);
    \draw   [dotted] [color=purple][line width=1.5pt] (a) -- (v5);
    \draw (v5) -- (c);

    \end{tikzpicture}

%% file: sectionTreeDecAndBranches.tex
A \emph{tree decomposition} \cite[p.~337]{Diestel10} of a graph~$G$
is a pair~$(T, \calV)$, consisting of a tree~$T$ and a collection
$\calV = \{ V_t: t \in V(T)\}$ of (different)
\emph{bags}~$V_t \subseteq V(G)$,
that satisfies the following three conditions:
\begin{itemize}
\item $\bigcup_{t \in V(T)} V_t = V(G);$
\item for every~$uv \in E(G)$, there exists
  a bag~$V_t$ such that~$u,v \in V_t;$
\item if a vertex~$v$ is in two different bags~$V_{t_1}, V_{t_2}$,
  then~$v$ is also in any bag~$V_t$ such that~$t$ is on the (unique)
  path from~$t_1$ to~$t_2$ in~$T$.
\end{itemize}

The \emph{treewidth}~$tw(G)$ is the number
$\min \{ \max\{|V_t|-1: t\in V(T)\}: (T, \calV)$ is a
tree decomposition of $G  \}$.
We refer to the vertices of~$T$ as \emph{nodes}.

If~$G$ is a graph with treewidth~$k$,
then we say that~$(T, \mathcal{V})$ is a \emph{full tree decomposition}\index{full tree decomposition}
of~$G$ if~$|V_t|=k+1$ for every~$t\in V(T)$,
and~$|V_t \cap V_{t'}|=k$ for every~$tt' \in E(T)$.
\begin{proposition}[{\cite[Lemma 8]{Bodlaender98}}{\cite[Theorem 2.6]{Gross14}}]\label{prop:full-tree-dec}
Every graph has a full tree decomposition.
\end{proposition}

Let~$G$ be a graph and~$(T, \calV)$ be a tree decomposition of~$G$.
Given two different nodes~$t,t' \in V(T)$, we
denote by~$\Branch_t(t')$
the component of~$T-t$ where~$t'$ lies.
We say that such component is a \emph{branch}\index{branch} of~$T$ at~$t$
and that the components
of~$T-t$ are the \emph{branches} of~$T$ at~$t$~\cite{Heinz13}.
Similarly, for a vertex~${v \notin V_t}$, it is denoted by 
$\Branch_t(v)$
the branch~$\Branch_t(t')$ of~$T$ at~$t$ 
such that~$v \in V_{t'}$.
In that case, we also say that~${v \in \Branch_t(t')}$
or that~$v$ is \emph{in}~$\Branch_t(t')$.

Let~$t \in V(T)$.
Let~$C'$ be a path or cycle in~$G$ fenced by~$V_t$.
It is easy to see that, for every $u,v \in V(C')$,
we have $\Branch_t(u)=\Branch_t(v)$.
Hence, when $V(C') \not \subseteq V_t$, we say that
$\Branch_t(C')=\Branch_t(v)$. And, if $V(C') \subseteq V_t$,
we say that $\Branch_t(C')=(\emptyset,\emptyset)$, that is,
a subtree of $T$ with empty set of nodes and edges.

The next proposition relates the concepts of separation and branches.
\begin{proposition}[{\cite[Lemma~12.3.1]{Diestel10}}]\label{prop:core-sep-tt'}
  Let~$tt' \in E(T)$. 
  Let~$u,v \in V(G)$ be such that~${u \notin V_{t}}$ and~${v \notin V_{t'}}$.
  If~$u \in \Branch_t(t')$ and~$v \in \Branch_{t'}(t)$, then~$V_{t} \cap V_{t'}$
  separates~$u$ and~$v$.
\end{proposition}

A~$k$-\emph{clique} in a graph is a set
of $k$ pairwise adjacent vertices.
A~$k$-\emph{tree}\index{$k$-tree} is defined recursively as follows.  
The complete graph on~$k$ vertices is a~$k$-tree.
Any graph obtained from a~$k$-tree by adding a new
vertex and making it adjacent to exactly 
all the vertices of an existing~$k$-clique is also a~$k$-tree.
A graph~$G$ is a \emph{partial} $k$-\emph{tree}\index{partial $k$-tree} if and only if
$G$ is the subgraph of a~$k$-tree. 
Partial~$k$-trees are closely related to the definition
of tree decomposition.
In fact, a graph~$G$ is a partial~$k$-tree if and only if~${\tw(G) \leq k}$ \cite[Theorem 35]{Bodlaender98}.

In what follows, we fix a 2-connected partial 3-tree~$G$ such that~$\tw(G)=3$,
and a full tree decomposition~$(T, \calV)$ of~$G$.

For every~$t\in V(T)$, and for every triple of vertices~$\Delta$
in~$V_t$, it is denoted by~$\cB_{t}(\Delta)$ the union of the branches
of the neighbors of~$t$ in~$T$ such that the corresponding bag
contains~$\Delta$.
That is,
$$\cB_{t}(\Delta)=\bigcup \{ \Branch_{t}(t'): tt'\in E(T) \mbox{ and } \Delta \subseteq V_{t'}\}.$$
For a vertex~$v \in V(G)$,
we say that~$v$ is~$t$-\emph{inside}~$\Delta$\index{$t$-inside vertex}
if~$v \in V_{t'}$ for some~$t' \in V(T)$ with~$\Branch_{t}(t') \subseteq \cB_{t}(\Delta)$.
Otherwise, we say that~$v$  is~$t$-\emph{outside}~$\Delta$.\index{$t$-outside vertex}
When the context is clear, we just say that~$v$ is~$\emph{inside}$\index{inside vertex} 
or~$\emph{outside}$\index{outside vertex}~$\Delta$ (Figure \ref{fig:lct-tw3-inside-outside}).

\begin{figure}[ht]
\centering
\includegraphics[width=.8\textwidth]{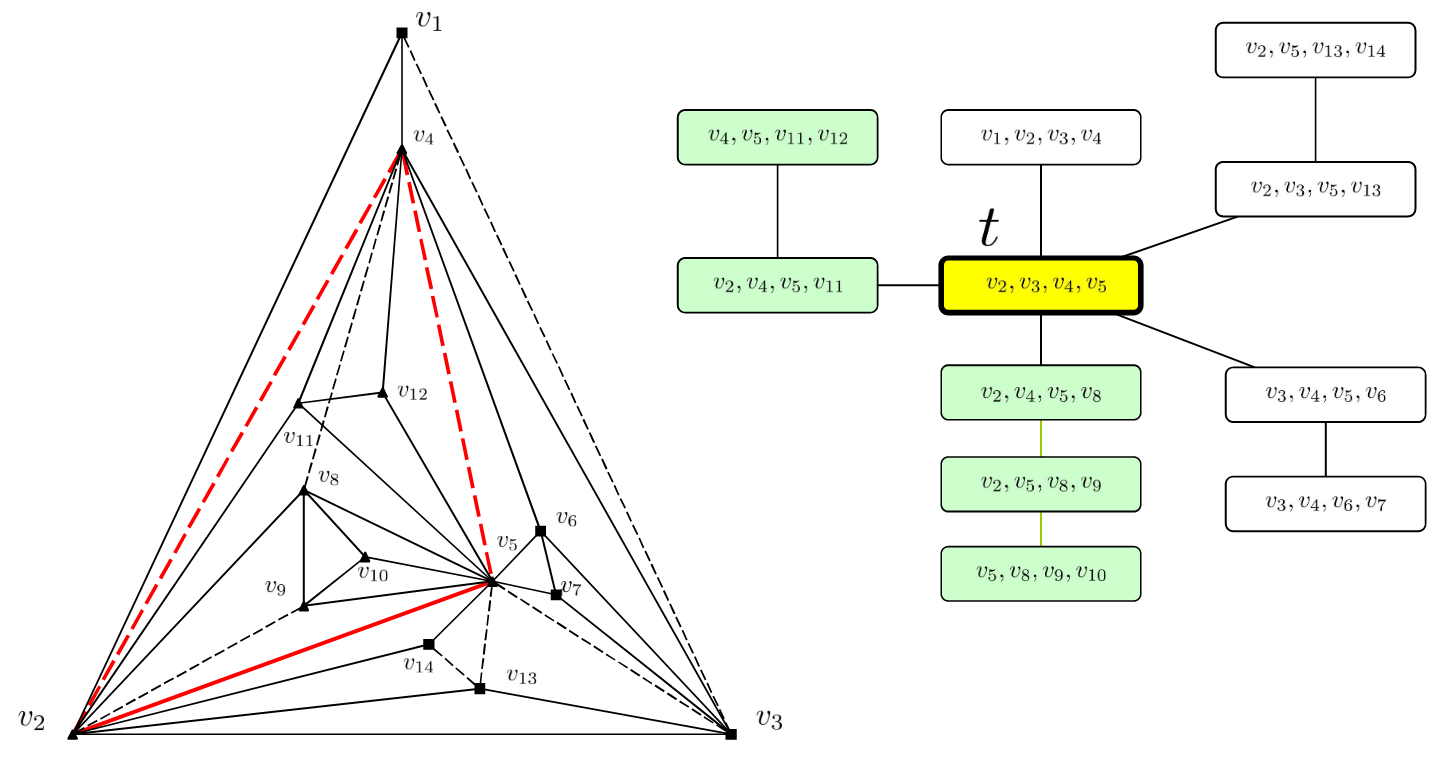}
\caption{A partial 3-tree of treewidth 3 and a corresponding full tree decomposition.
The dashed lines imply that there is no edge between the corresponding vertices.
If~$t$ is the node with~$V_t=\{v_2,v_3,v_4,v_5\}$ and~$\Delta=\{v_2,v_4,v_5\}$,
then the vertices~$v_2,v_4,v_5,v_8,v_9,v_{10},v_{11},v_{12}$ are~$t$-inside~$\Delta$ and 
all the other vertices are~$t$-outside~$\Delta$.
Also, the cycles~$v_2v_5v_8v_2$ and~$v_4v_{12}v_5v_{11}v_4$ are~$t$-inside~$\Delta$,
the cycle~$v_5v_7v_3v_4v_6v_5$ is~$t$-outside~$\Delta$, and
the cycles~$v_2v_5v_{14}v_2$ and~$v_3v_4v_{11}v_5v_6v_3$~$t$-jump~$\Delta$.}
\label{fig:lct-tw3-inside-outside}
\end{figure}

Given a path~$C'$ that 2-intersects~$\Delta$ at the extremes of~$C'$,
we say that~$C'$ 
is~$t$-\emph{inside}~$\Delta$\index{$t$-inside path}
if every vertex of~$C'$ is~$t$-inside~$\Delta$,
and~$C'$ is~$t$-\emph{outside}~$\Delta$\index{$t$-outside path} otherwise.
Given a cycle~$C$ that intersects~$\Delta$ at least twice,
we say that~$C$ is~$t$-\emph{inside}~$\Delta$\index{$t$-inside cycle}
if every~$\Delta$-part of~$C$ is~$t$-\emph{inside}~$\Delta$,
$C$ is~$t$-\emph{outside}~$\Delta$\index{$t$-outside cycle}
if every~$\Delta$-part of~$C$ is~$t$-\emph{outside}~$\Delta$, 
and~$C$~$t$-\emph{jumps}~$\Delta$\index{jump}
if~$C$ has a~$\Delta$-part~$t$-inside~$\Delta$
and a~$\Delta$-part~$t$-outside~$\Delta$.

Again, if the context is clear, we just say that~$C$
is~$\emph{inside}$,~$\emph{outside}$, or \emph{jumps}~$\Delta$
(Figure~\ref{fig:lct-tw3-inside-outside}).
Furthermore, we abuse notation and, when the context is clear, for
a cycle~$C$ that jumps~$\Delta$,
we say that~$C$~$\ell$-jumps~$\Delta$ if~$|V(C) \cap \Delta| = \ell$.
Observe that~$\ell \in \{2,3\}$.
Moreover, if~$C$~$\ell$-jumps~$\Delta$,
we say that~$C$~$\ell$-jumps~$\Delta$ at~$V(C) \cap \Delta$.
Throughout this chapter, we abbreviate a triple of vertices~$\{a,b,c\}$
by~$abc$.


%% file: sectionMainTheorem.tex
The next lemma conceals the heart of the proof of Theorem~\ref{theorem:lct-tw3-lct=1}.
The proof of that lemma is presented in Section~\ref{section:mainlemma}.
Throughout this and the following sections, for a graph $G$, we let $L(G)$ be the length of a longest cycle in $G$.

\begin{lemma} \label{lemma:lct-tw3-lct=1-or-fenced}
  Let~$G$ be a 2-connected graph with treewidth~$3$.
  Let~$(T, \calV)$ be a full tree decomposition of~$G$
  and~$V_t \in \calV$.
  Then either~$\lct(G)=1$ or there exists a longest cycle in~$G$ that is fenced by~$V_t$
  and intersects~$V_t$ at most three times.
\end{lemma}

Using this lemma, we derive the main result of this section.

\begin{theorem}\label{theorem:lct-tw3-lct=1}
If~$G$ is a 2-connected partial 3-tree, then~$\lct(G)=1$.
\end{theorem}
\begin{proof}
If~$\tw(G) \leq 2$ then~${\lct(G)=1}$ by \cite[Theorem 2]{Gutierrez18}.
Therefore, we may assume that~${\tw(G) = 3}$.
Assume by contradiction that~${\lct(G)>1}$.
It is straightforward to see that
${L:=L(G) \geq 5}$.
Let~$(T, \calV)$ be a full tree decomposition of~$G$,
which exists by Proposition~\ref{prop:full-tree-dec}.
  For every~${t \in V(T)}$,
  let~$\calF(t)$ be the set of longest cycles in~$G$
  that are fenced by~$V_t$ and intersect~$V_t$ at most three times.
  By Lemma~\ref{lemma:lct-tw3-lct=1-or-fenced},~${\calF(t) \neq \emptyset}$
  for every~${t \in V(T)}$.
  Observe that, as~$L \geq 5$, for every such~$t$,
  no cycle in~$\calF(t)$
  is contained in~$G[V_t]$.
  
  We direct some of the edges of $T$ to create an auxiliary directed forest $T'$ in the following way: $tt' \in E(T')$ if and only if $tt' \in E(T)$ and there exists a cycle $C \in \mathcal{F}(t)$
with $\Branch_t(C)=\Branch_t(t')$.
By taking the last arc $tt'$ of a maximal directed path in $T'$, we will have
two longest cycles~${C \in \calF(t)}$
and~${D \in \calF(t')}$ 
such that~${\Branch_t(C)=\Branch_t(t')}$ and~${\Branch_{t'}(D)=\Branch_{t'}(t)}$.

Note that the bags containing vertices of~$C$ are 
in~$\Branch_t(t') \cup \{t\}$, and the bags containing
vertices of~$D$ are in~$\Branch_{t'}(t) \cup \{t'\}$.
Since~$\Branch_t(t')$ and~$\Branch_{t'}(t)$ are disjoint, 
${V(C) \cap V(D) \subseteq V_{t} \cup V_{t'}}$.  
Let~${V_t=\{a,b,c,u\}}$ and~${V_{t'}=\{a,b,c,w\}}$.

It is straightforward to see that
$u \notin V(C)$ and that $w\notin V(D)$.
and therefore
~$V(C) \cap V(D) \subseteq V_{t}\cap V_{t'}$.
As~${|V(C) \cap V(D)| \geq 2}$,
this implies that~${|V(C) \cap V_{t}|\geq 2}$
and~${|V(D) \cap V_{t'}|\geq 2}$.
  Observe that any other longest cycle intersects~${V_t \cap V_{t'}}$ at least twice.
  Otherwise it would intersect~$C$ or~$D$ at most once, a contradiction.
  As~$\lct(G)>1$, there exists a longest cycle~$F$ that does not contain~$a$ and a longest cycle~$H$ that does not contain~$c$. By the previous
  observation,~$F$ intersects~${V_t \cap V_{t'}}$ at~$\{b,c\}$
  and~$H$ intersects~${V_t \cap V_{t'}}$ at~$\{a,b\}$.
  Let~$F'$ and~$F''$ be the two~$bc$-parts of~$F$, with~$|F'|\geq |F''|$.
   Let~$H'$ and~$H''$ be the two~$ab$-parts of~$H$, with~$|H'|\geq |H''|$.
  The rest of the proof is divided into whether~$C$ and~$D$ 2-intersect or 3-intersect~${V_t \cap V_{t'}}$.
  
  \setcounter{case}{0}
  
  \begin{case}~$C$ {2-intersects}~$V_t \cap V_{t'}$.
  
  Without loss of generality, assume that~$C$ {2-intersects}~$V_t \cap V_{t'}$ at~$\{a,b\}$.
  Suppose for a moment that~$F'$ is internally disjoint from~$C$.
  As~$F$ and~$C$ intersect each other at least twice,~$F''$ and~$C$ intersect each other
  at a vertex~$x$ different from~$b$.
  Let~$F''_{cx}$ be the subpath of~$F''$ that is internally disjoint from~$C$, starts at~$c$ and ends at~$x$.
  Let~$C_{xb}$ be the subpath of~$C$ with extremes~$x$ and~$b$ such 
  that~${|C_{bx}| \geq L/2}$. Then~${F' \cdot F''_{cx} \cdot C_{xb}}$ is a cycle longer than~$L$, a contradiction
  (Figure~\ref{fig:lct=1-partial-3-trees}(a)).
  We conclude that~$F'$ is internally disjoint from~$D$.
  If~$D$ {2-intersects}~$V_t \cap V_{t'}$, then,
  as~$|V(C) \cap V(D)| \geq 2$, we must have that~$D$ intersects~$V_t \cap V_{t'}$
  at~$\{a,b\}$.
  But then we obtain a contradiction as before, with~$D$ instead of~$C$.
  So~$D$ {3-intersects}~$V_t \cap V_{t'}$ (at~$\{a,b,c\}$).
    Let~$D_{ab},D_{ac},D_{bc}$ be the~$abc$-parts of~$D$.
  As~${C'' \cdot D_{ab}}$ and~${(D-D_{ab}) \cdot C'}$ are both longest cycles, we have that
 ~${|C'|=|D_{ab}|}$. Similarly,~${|C''|=|D_{ab}|}$.
  Hence~${|C'|=|C''|=|D_{ab}|=L/2}$.
  As~$F'$ is internally disjoint from~$D$,
  ~${D_{ab} \cdot F' \cdot D_{ca}}$ is a cycle longer than~$L$, again a contradiction
  (Figure~\ref{fig:lct=1-partial-3-trees}(b)).
  \end{case}    
  
  \begin{case} Both~$C$ and~$D$ 3-intersect~$V_t \cap V_{t'}$.
  
   Let~$C_{ab},C_{ac},C_{bc}$ be the~$abc$-parts of~$C$.
  Let~$D_{ab},D_{ac},D_{bc}$ be the~$abc$-parts of~$D$.
  Assume without loss of generality that~$F'$ is internally disjoint from~$C$.
  Observe that both~$(C-C_{bc}) \cdot F'$ and~$C_{bc} \cdot F'$ are cycles.
  We conclude that~$|C_{bc}|=|C_{ab}|+ |C_{ac}|= |F'| = |F''|=L/2$.
  As~$(C-C_{bc}) \cdot D_{bc}$ and~$(D-D_{bc}) \cdot C_{bc}$ are cycles,
 ~$|C_{bc}| =|D_{bc}| =L/2$. A similar analysis with~$H$ instead of~$F$ shows us that 
 ~$|C_{ab}| =|D_{ab}| =L/2$.
  This implies that~$|C_{ac}|=0$, a contradiction (Figure~\ref{fig:lct=1-partial-3-trees}(c)).  
  \end{case}
This concludes the proof of the theorem.
\end{proof}

\begin{figure}[H]
\centering
\scalebox{.8}{
\subfigure[ ]{
 \begin{tikzpicture}
    \node at (3,2) {$V_t$};
    \node at (6,2) {$V_{t'}$};
    \node at (4.5,1)[myblue,label=above :$a$,fill=myblue, circle](a) {};
    \node at (4.5,0)[myblue,label=above :$b$,fill=myblue, circle](b) {};
    \node at (4.5,0.02) (b1) {};
     \node at (4.5,-1)[myblue,label=above :$c$,fill=myblue, circle](c) {};
      
    \draw (4,0) ellipse (1 and 2);
    \draw (5,0) ellipse (1 and 2);
    
    \draw [color=blue] [line width=1.0pt, style =dashed] (a) .. controls (7, 1) .. (b);
     \draw [color=blue] [line width=1.0pt, style =dashed] (a) .. controls (9, 1.5) .. (b);
    
    \node at (7.5,1.75) [color=blue]{$C$};
    
    \node at (2,-1) [color=mygreen]{$F'$};
    
    \node at (7,0.8)[myblue,label=below right:$x$,fill=myblue, circle](x) {};


    \draw [color=mygreen] [line width=1.0pt,densely dashdotdotted] (b) .. controls (1, -2) .. (c);
    \draw [color=mygreen] [line width=1.0pt,densely dashdotdotted] (x) .. controls (6, -0.5) .. (c);
    
    \end{tikzpicture}
}}
\scalebox{.8}{
\subfigure[ ]{
 \begin{tikzpicture}
    \node at (3,2) {$V_t$};
    \node at (6,2) {$V_{t'}$};
    \node at (4.5,1)[myblue,label=above :$a$,fill=myblue, circle](a) {};
    \node at (4.5,0)[myblue,label=above :$b$,fill=myblue, circle](b) {};
    \node at (4.5,0.02) (b1) {};
     \node at (4.5,-1)[myblue,label=above :$c$,fill=myblue, circle](c) {};
      
    \draw (4,0) ellipse (1 and 2);
    \draw (5,0) ellipse (1 and 2);
    
    \draw [color=blue] [line width=1.0pt, style =dashed] (a) .. controls (7, 1) .. (b);
     \draw [color=blue] [line width=1.0pt, style =dashed] (a) .. controls (9, 1.5) .. (b);
    
    \node at (7.5,1.75) [color=blue]{$C$};
    
    \draw [color=red] [line width=1.0pt] (a) .. controls (1.5, 0.15) .. (b);
    \draw [color=red] [line width=1.0pt] (b) .. controls (1.5, -0.15) .. (c);
    \draw [color=red][line width=1.0pt] (a) .. controls (0,0) .. (c);
    \node at (1.5,0.75) [color=red]{$D$};
    \node at (7,-1) [color=mygreen]{$F'$};
    


    \draw [color=mygreen] [line width=1.0pt,densely dashdotdotted] (b) .. controls (8, -2) .. (c);
    
    \end{tikzpicture}
}}
\scalebox{.8}{
\subfigure[ ]{
\begin{tikzpicture}
    \node at (3,2) {$V_t$};
    \node at (6,2) {$V_{t'}$};
    \node at (4.5,1)[myblue,label=above :$a$,fill=myblue, circle](a) {};
    \node at (4.5,0)[myblue,label=above :$b$,fill=myblue, circle](b) {};
    \node at (4.5,0.02) (b1) {};
     \node at (4.5,-1)[myblue,label=above :$c$,fill=myblue, circle](c) {};
      
    \draw (4,0) ellipse (1 and 2);
    \draw (5,0) ellipse (1 and 2);
    
    \draw [color=blue] [line width=1.0pt,dashed] (a) .. controls (7.5, 0.15) .. (b);
    \draw [color=blue] [line width=1.0pt,dashed] (b) .. controls (7.5, -0.15) .. (c);
    \draw [color=blue,dashed][line width=1.0pt] (a) .. controls (9,0) .. (c);
    \node at (7.5,0.75) [color=blue]{$C$};

    
    \draw [color=red] [line width=1.0pt] (a) .. controls (1.5, 0.15) .. (b);
    \draw [color=red] [line width=1.0pt] (b) .. controls (1.5, -0.15) .. (c);
    \draw [color=red][line width=1.0pt] (a) .. controls (0,0) .. (c);
    \node at (1.5,0.75) [color=red]{$D$};



    \draw [color=mygreen] [line width=1.0pt,densely dashdotdotted] (b) .. controls (1, -2) .. (c);
    \node at (2,-1) [color=mygreen]{$F'$};
    
   \draw [color=brown] [line width=1.0pt,densely dashdotdotted] (a) .. controls (7.5, 2) .. (b);
    \node at (7.2,1.5) [color=brown]{$H'$};
    \end{tikzpicture}
    }}
\caption{Cycles~$C,D,F$ and~$H$ as in the proof of
Theorem~\ref{theorem:lct-tw3-lct=1}.
(a)~$C$ {2-intersects}~$V_t \cap V_{t'}$ and~$F'$ is internally disjoint from~$C$.
(b)~$C$ {2-intersects}~$V_t \cap V_{t'}$ and~$F'$ is internally disjoint from~$D$.
(c)~$C$ {3-intersects}~$V_t \cap V_{t'}$.}
\label{fig:lct=1-partial-3-trees}
\end{figure}
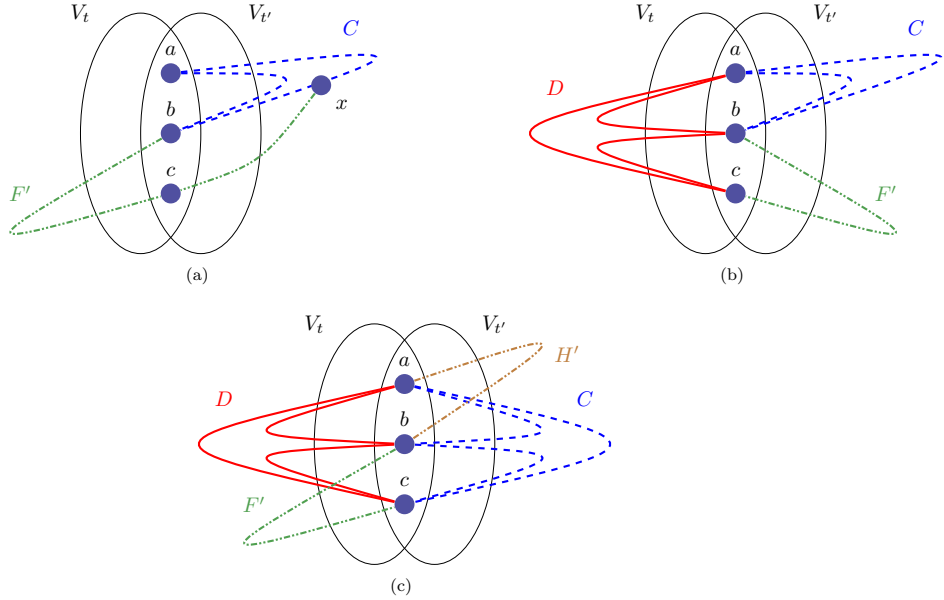

%% file: sectionMainLemma.tex
For all lemmas in this subsection,
we fix a graph~$G$ with~$\tw(G)=3$ and~${\lct(G)>1}$,
a full tree decomposition~$(T,\calV)$ of~$G$, and
a node~$t \in V(T)$.
We let~$\twocross$ be the set of all longest cycles in~$G$
that 2-cross~$V_t$, and let~$\calF$ be the set of
all longest cycles in~$G$ that are fenced by~$V_t$
and intersect~$V_t$ at most three times.
Therefore, proving Lemma~\ref{lemma:lct-tw3-lct=1-or-fenced}
reduces to prove that~$\calF \neq \emptyset$.

\begin{proof}[Proof of Lemma~\ref{lemma:lct-tw3-lct=1-or-fenced}]
Suppose that~$\twocross \neq \emptyset$.
If there are two non~$V_t$-equivalent longest cycles in~$\twocross$,
then we are done by Lemma~\ref{lemma:lct-tw3-X2-2-nonequivalent-then-F}.
If all cycles in~$\twocross$ are~$V_t$-equivalent, then we are done
by Lemma~\ref{lemma:lct-tw3-X2-all-equivalent-then-F}.
Hence~$\twocross = \emptyset$, and we
are done by Lemma~\ref{lemma:lct-tw3-X2-empty-then-F}.
\end{proof}

Next we prove Lemmas~\ref{lemma:lct-tw3-X2-2-nonequivalent-then-F},~\ref{lemma:lct-tw3-X2-all-equivalent-then-F} and~\ref{lemma:lct-tw3-X2-empty-then-F}.
We depend on an auxiliary
result (Corollary~\ref{corollary:lct-tw3-lct>1-implies-not-jump})
presented and proved ahead, in Section~\ref{section:auxiliarylemmas}.

\begin{lemma}
\label{lemma:lct-tw3-X2-2-nonequivalent-then-F}
If there are two non~$V_t$-equivalent longest cycles in~$\twocross$,
then~${\calF \neq \emptyset}$.
\end{lemma}
\begin{proof}
Let~$V_t=\{a,b,c,d\}$.
Suppose by contradiction that~${C,D \in \twocross}$,
${V(C) \cap V_t =\{a,b\}}$ and~${V(D) \cap V_t \neq \{a,b\}}$.
Let~$C'$ and~$C''$
be the corresponding~$ab$-parts of~$C$.
If~$D$ 2-crosses~$V_t$ at~$\{c,d\}$,
then~$V(C) \cap V(D) =\emptyset$, a contradiction.
Hence, we may assume that~$D$ 2-crosses~$V_t$ at~$\{a,c\}$.
Let~$D'$ and~$D''$
be the corresponding~$ac$-parts of~$D$.
As both~$C$ and~$D$ cross~$V_t$, 
we may assume, without loss of generality, that
$C'$ is internally disjoint from~$D'$ and that
$C''$ is internally disjoint from~$D''$.
As~$\lct(G)>1$, there exists a longest cycle~$F$
that does not contain~$a$.
If~$F$ is fenced by~$V_t$, we are done, so let us assume
that~$F$ crosses~$V_t$.

First consider the case in which~$F$ {2-intersects}~$V_t$.
As $F$ intersects both $C$ and $D$,
we have that $F$ {2-intersects}~$V_t$ at~$\{b,c\}$.
Let~$F'$ and~$F''$ be the corresponding~$bc$-parts of~$F$.
Suppose that~$C$, $D$ and~$F$ are~$t$-inside~$abc$.
As~$G$ is 2-connected, 
there exist two internally disjoint
paths starting at~$d$ and ending at distinct vertices in
$\{a,b,c\}$.
Let~$P$ and~$Q$ be these two paths, and suppose
without loss of generality that~$b$ is an extreme of~$P$ 
and that~$c$ is an extreme of~$Q$.
By Proposition \ref{prop:core-sep-tt'},
both~$P$ and~$Q$ are internally disjoint
from~$C$ and~$D$.
But then~${P \cdot Q \cdot D' \cdot C'}$ and
${P \cdot Q \cdot D'' \cdot C''}$ are both cycles, 
at least one of them longer than~$L$,
a contradiction (Figure~\ref{fig:lct-2-crosses-are-equivalent-1}(a)).
Hence, at least one of~$\{C,D,F\}$, say $C$, is
not~$t$-inside~$abc$.
We may assume that~$C''$ is internally disjoint from both~$D$ and~$F$.
As~$D$ and~$F$ 2-cross~$V_t$,
we may assume that~$D'$ is internally disjoint from~$F'$
and that~$D''$ is internally disjoint from~$F''$.
But then~${C'' \cdot D' \cdot F'}$ and
${C'' \cdot D'' \cdot F''}$ are both cycles, 
at least one of them longer than~$L$,
a contradiction
(Figure~\ref{fig:lct-2-crosses-are-equivalent-1}(b)).

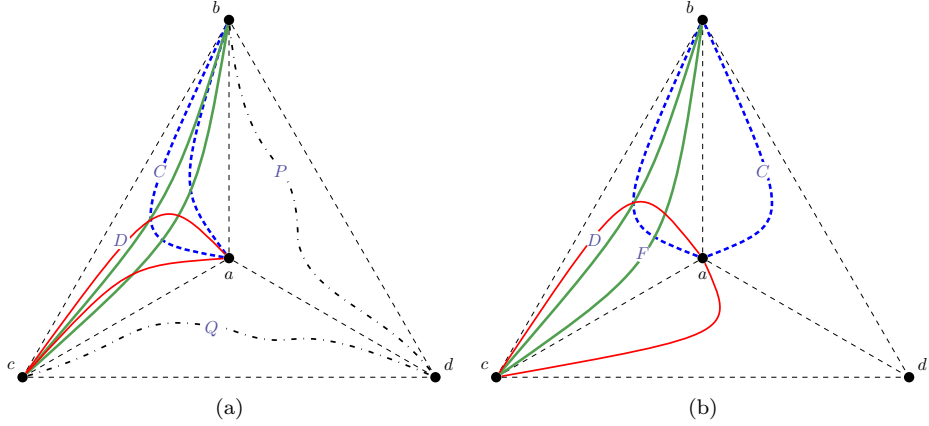
\begin{figure}[H]
\centering
\subfigure[ ]{
\resizebox{.5\textwidth}{!}{\input{Figures/lct-2-crosses-are-equivalent-1}}
}
\subfigure[ ]{
\resizebox{.5\textwidth}{!}{\input{Figures/lct-2-crosses-are-equivalent-2}}
}
\caption{$(a)$~$C,D$ and~$F$ are~$t$-inside~$abc$.
$(b)$ At least one of~$\{C,D,F\}$ is not~$t$-inside~$abc$.
}
\label{fig:lct-2-crosses-are-equivalent-1}
\end{figure}

Now suppose that~$F$ {3-intersects}~$V_t$ (at~$\{b,c,d\}$).
If both~$C$ and~$D$ are~$t$-inside~$abc$, then we can proceed
as in the previous paragraph with~$F_{bd}$ and~$F_{cd}$ instead
of~$P$ and~$Q$ respectively.
If both~$C$ and~$D$ are~$t$-outside~$abc$,
then
~$V(C) \cap V(D)=\{a\}$, a contradiction.
Hence, at least one of~$\{C,D\}$~$t$-jumps~$abc$.
Without loss of generality, suppose that~$C$~$t$-jumps~$abc$
and that~$\Branch_t(C'') \subseteq \cB_{t}(abd)$.
Let~$R$ be the subpath of~$F_{bd}$ that starts at~$d$,
ends at a vertex of~$C''$ and is internally disjoint from~$C''$.
Suppose that the other extreme of~$R$ is~$x$.
First consider the case in which~$D$ also~$t$-jumps~$abc$,
and without loss of generality suppose that  
$\Branch_t(D') \subseteq \cB_{t}(acd)$.
Let~$S$ be the subpath of~$F_{cd}$ that starts at~$d$,
ends at~$D'$ and is internally disjoint from~$D'$.
Suppose that the other extreme of~$S$ is~$y$.
Let~$C''_{ax}$ and~$C''_{bx}$ be the two~$x$-tails of~$C''$.
Let~$D'_{ay}$ and~$D'_{cy}$ be the two~$y$-tails of~$D'$.
Then~$C'\cdot C''_{bx} \cdot  R \cdot S \cdot D'_{ya}$
and ~$D''\cdot D'_{cy} \cdot S \cdot R \cdot C''_{xa}$
are cycles, at least one of them longer than~$L$, a contradiction
(Figure~\ref{fig:lct-2-crosses-are-equivalent-2}(a)).
Finally consider the case in which~$D$ is~$t$-inside~$abc$.
Then~$C'\cdot C''_{bx} \cdot  R \cdot F_{dc} \cdot D'$
and~$D''\cdot C''_{ax} \cdot R \cdot F_{dc}$
are cycles, at least one of them longer than~$L$, a contradiction
(Figure~\ref{fig:lct-2-crosses-are-equivalent-2}(b)).
\end{proof}

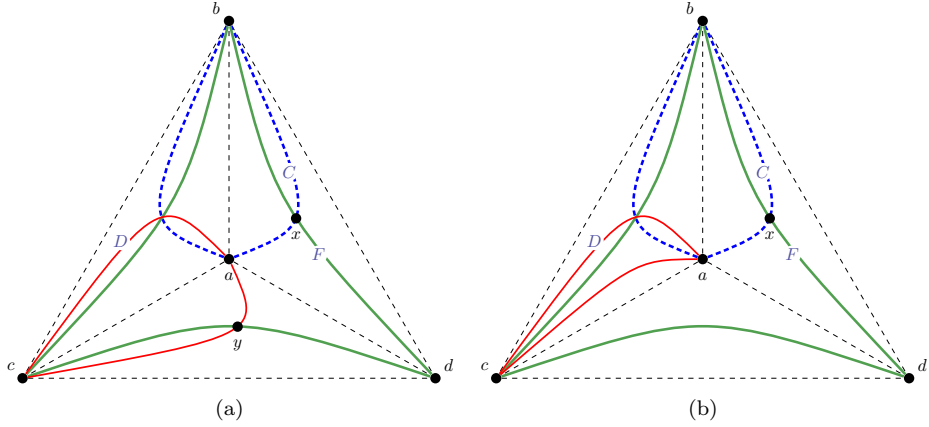
\begin{figure}[H]
\centering
\subfigure[ ]{
\resizebox{.5\textwidth}{!}{\input{Figures/lct-2-crosses-are-equivalent-3}}
}
\subfigure[ ]{
\resizebox{.5\textwidth}{!}{\input{Figures/lct-2-crosses-are-equivalent-4}}
}
\caption{Cases in the proof of Lemma~\ref{lemma:lct-tw3-X2-2-nonequivalent-then-F}.
$F$ {3-intersects} $V_t$.
$(a)$~$D$~$t$-jumps~$abc$.
$(b)$~$D$ is~$t$-inside~$abc$.
}
\label{fig:lct-2-crosses-are-equivalent-2}
\end{figure}

\begin{lemma}
\label{lemma:lct-tw3-X2-all-equivalent-then-F}
  If~$\twocross \neq \emptyset$
  and all cycles in~$\twocross$ are~$V_t$-equivalent,
  then~$\calF \neq \emptyset$.
\end{lemma}
\begin{proof}
Let~$V_t=\{a,b,c,d\}$.
Let~$C \in \twocross$.
Assume without loss of generality that~${V(C) \cap V_t = \{a,b\}}$.
As~${\lct(G)>1}$, there exists a longest cycle~$D$ that does not contain~$a$,
and a longest cycle~$F$ that does not contain~$b$.
If either~$D$ or~$F$ is fenced by~$V_t$,~$\calF\neq \emptyset$.
So, as all cycles in~$\twocross$ are~$V_t$-equivalent,
both~$D$ and~$F$ 3-intersect~$V_t$.

Observe that~$C$ must-jump both~$abc$ and~$abd$.
Indeed, otherwise, we can join~$D$ and~$F$ with a~$ab$-part of~$C$. In fact, this is valid to any cycle equivalent to $C$, thus we have the following Claim.
\begin{claim}\label{remark:lct-tw3-remark-C-jumps-abc-abd}
Every cycle that 2-crosses~$V_t$ at~$\{a,b\}$ jumps both~$abc$ and~$abd$.
\end{claim}

Let~$C'$ and~$C''$ be the two~$ab$-parts of~$C$.
By Claim~\ref{remark:lct-tw3-remark-C-jumps-abc-abd},
we may assume that~$C'$ is~$t$-inside~$abc$ and
that~$C''$ is~$t$-inside~$abd$.

Moreover, 
we have the following Claim regarding cycles $D$ and $F$.

\begin{claim}\label{remark:lct-tw3-remark-1}
$D$ and~$F$ jump both~$abc$ and~$abd$.
\end{claim}
\begin{proof}
[Proof of Claim~\ref{remark:lct-tw3-remark-1}]
If~$D$ does not jump at least one of~$\{abc,abd\}$, then~$C$ and~$D$
only intersect at~$b$, a contradiction to the fact that~$G$ is 2-connected.
Hence~$D$ jumps at least one of~$\{abc,abd\}$.
Analogously,~$F$ jumps at least one of~$\{abc,abd\}$.
Suppose by contradiction that the claim is not true.
Then, one of~$\{D,F\}$, say~$D$,
jumps exactly one of~$\{abc,abd\}$, say~$abc$.
First suppose that~$F$ jumps both of~$\{abc,abd\}$.
Recall that~$C''$ is the~$ab$-part of~$C$ that is~$t$-inside~$abd$.
Let~$R$ be the subpath of~$C''$ that is internally disjoint from~$F$,
starts at~$b$ and ends at a vertex~$x$ of~$F$ (possibly~$x=a$).
Then~${F_{cd} \cdot F_{dx} \cdot R \cdot D_{bc}}$
and~${D_{cd} \cdot D_{db} \cdot R \cdot F_{xa} \cdot F_{ac}}$
are cycles, one of them longer than~$L$, a contradiction
(Figure~\ref{fig:2crosses-implies-fenced-1}(a)).

The rest of the proof follows by similar analysis to the previous case, we show the cases in
Figures~\ref{fig:2crosses-implies-fenced-1}(b)
to (d).
\end{proof}

\begin{figure}[H]
\centering

\subfigure[ ]{
\resizebox{.22\textwidth}{!}{\input{Figures/lct-2-crosses-implies-fenced-1}}
}
\subfigure[ ]{
\resizebox{.22\textwidth}{!}{\input{Figures/lct-2-crosses-implies-fenced-2}}
}
\subfigure[ ]{
\resizebox{.22\textwidth}{!}{\input{Figures/lct-2-crosses-implies-fenced-3}}
}
\subfigure[ ]{
\resizebox{.22\textwidth}{!}{\input{Figures/lct-2-crosses-implies-fenced-4}}
}

\caption{Cases in the proof of
Claim~\ref{remark:lct-tw3-remark-1}
of Lemma~\ref{lemma:lct-tw3-X2-all-equivalent-then-F}.
(a)~$F$ jumps both~$abc$ and~$abd$.
(b)~$F$ jumps only~$abd$ and~$F_{cd}$ does not jump~$acd$. In this case, $D_{cd} \cdot F_{cd} \in \mathcal{X}_2$.
(c)~$F$ jumps only~$abd$ and~$F_{cd}$ jumps~$acd$.
We consider a subpath $R$ of $C$ that starts at $b$ and finishes at $D$ and proceed similar to Case $a$.
(d)~$F$ jumps only~$abc$.
We can join parts of $F$ and $D$ with $C''$ and obtain longer cycles.
}
\label{fig:2crosses-implies-fenced-1}
\end{figure}
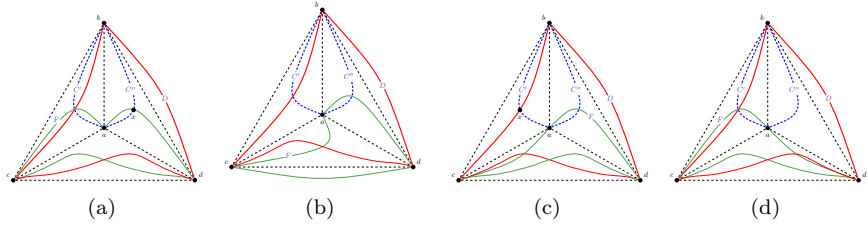

By Claim~\ref{remark:lct-tw3-remark-1}
we have
${\Branch_t(F_{ac}), \Branch_t(D_{bc}) \subseteq \cB_t(abc)}$ 
and~${\Branch_t(F_{ad}), \Branch_t(D_{bd}) \subseteq \cB_t(abd)}$
(Figure~\ref{fig:lct-remark-1}).

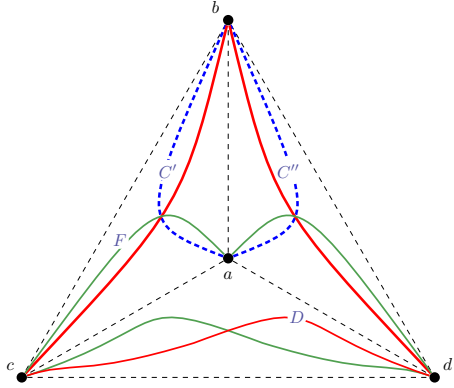
\begin{figure}[H]
\centering
\resizebox{.5\textwidth}{!}{\input{Figures/lct-remark-1}}
\caption{The structure given by Claim~\ref{remark:lct-tw3-remark-1} of Lemma~\ref{lemma:lct-tw3-X2-all-equivalent-then-F}.}
\label{fig:lct-remark-1}
\end{figure}

Now, by Corollary~\ref{corollary:lct-tw3-lct>1-implies-not-jump},
as~${\lct(G)>1}$, there exists a longest cycle~$H$ that does not~$t$-jump~$abc$.
If~$H$ intersects~$V_t$ at most once, then~$H$ is fenced by~$V_t$
and~$\calF \neq \emptyset$.
Also, by Claim \ref{remark:lct-tw3-remark-C-jumps-abc-abd}, $H$ can not 2-intersect $V_t$.
So,~$H$ intersects~$V_t$ at least three times.

Suppose for a moment that~$H$ {4-intersects}~$V_t$.
As~$H$ is a cycle, either~$H$ has a subpath
from~$a$ to~$c$ that is disjoint from~$b$ and~$d$,
or~$H$ has a subpath
from~$b$ to~$c$ that is disjoint from~$a$ and~$d$.
Without loss of generality assume the former.
Let~$H_{ac}$ be such subpath.
Then there exists a subpath~$R$ of~$H_{ac}$ from~$a$ to~$D_{cd}$
that is internally disjoint from both~$C$ and~$D$.
Let~$x$ be its extreme in~$D$.
Then~${D_{bc} \cdot C'' \cdot R \cdot D_{xc}}$ and
${C' \cdot D_{bd} \cdot D_{dx} \cdot R}$ are both cycles,
one of them longer than~$L$,
a contradiction
(Figure~\ref{fig:lct-2-crosses-implies-fenced-3}(a)).
We conclude that~$H$ {3-intersects}~$V_t$.
If~$H$ is~$t$-inside~$abc$, then~$H$ is fenced
by~$V_t$ by Corollary~\ref{corollary:lct-tw3-lct>1-implies-not-jump},
so~$\calF \neq \emptyset$ and we are done.
Thus,~$H$ is~$t$-outside~$abc$.

If~$H$ contains both~$a$ and~$c$, then~$H_{ac}$, the subpath of~$H$
from~$a$ to~$c$, is internally disjoint from~$C$ and we
proceed as before. Hence, either~$H$ does not contain~$a$
or~$H$ does not contain~$c$.
Suppose for a moment that~$H$ does not contain~$a$.
Then~$H_{bc}$, the subpath of~$H$ from~$b$ to~$c$,
is internally disjoint from~$C$ so we proceed
as before (with~$C$ and~$F$ instead of~$C$ and~$D$).
So, we may assume that~$H$ does not contain~$c$.
As~$H$ is~$t$-outside~$abc$, we have 
that~$H_{ab}$ is~$t$-inside~$abd$.
Moreover, observe that~$H$ is~$t$-inside~$abd$. Indeed, if for example
$H_{bd}$ is~$t$-outside~$abd$ then we proceed as before.
The same reasoning applies if~$H_{ad}$ is~$t$-outside~$abd$.
Thus, as~${(F-F_{ad}) \cdot H_{ad}}$ 
and~${(F-F_{ad}) \cdot H_{ab} \cdot H_{bd}}$ are cycles,
we have that~${|F_{ad}|\geq L/2}$.
Now, repeating the argument with~$abd$ instead of~$abc$, we
deduce that there exists a longest cycle~$J$ that {3-intersects}~$V_t$ at
$abc$ and is~$t$-inside~$abc$.
Hence~${|F_{ac}|\geq L/2}$. This implies that~$|F_{cd}|=0$, a contradiction
(Figure~\ref{fig:lct-2-crosses-implies-fenced-3}(b)).
\end{proof}

\begin{figure}[hbt]
\subfigure[ ]{
\resizebox{.5\textwidth}{!}{\input{Figures/lct-2-crosses-implies-fenced-5}}
}
\subfigure[ ]{
\resizebox{.5\textwidth}{!}{\input{Figures/lct-2-crosses-implies-fenced-6}}
}
\caption{Cases in the last part of the proof of Lemma~\ref{lemma:lct-tw3-X2-all-equivalent-then-F}.
(a)~$H$ {4-intersects}~$V_t$.
(b)~$H$ {3-intersects}~$V_t$.
}
\label{fig:lct-2-crosses-implies-fenced-3}
\end{figure}
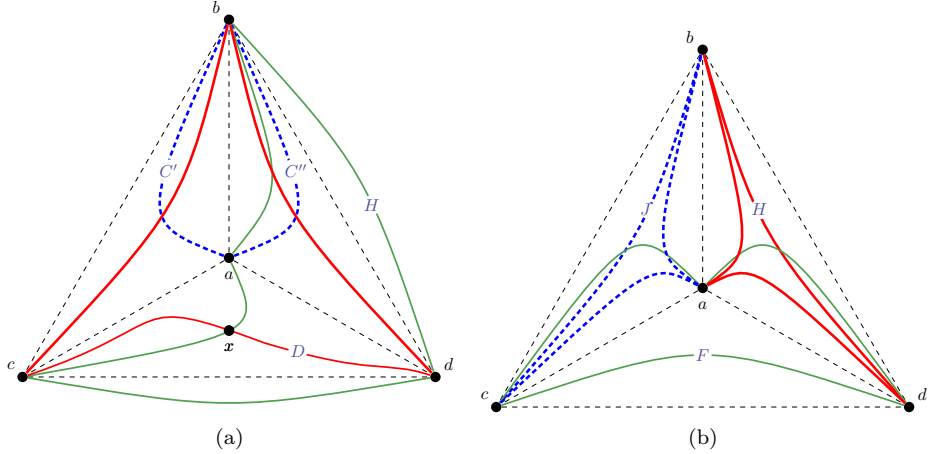

The next lemma is the most technical.
Its proof is organized into four claims.

\begin{lemma}\label{lemma:lct-tw3-3crosses}
\label{lemma:lct-tw3-X2-empty-then-F}
If~$\twocross = \emptyset$, then~$\calF \neq \emptyset$.
\end{lemma}
\begin{proof}
As~$\lct(G)>1$ and~${\twocross = \emptyset}$,
for every triple of vertices~${\Delta \subseteq V_t}$, there exists a nonempty
collection~$\cC_{\Delta}$ of longest cycles that 3-intersect~$V_t$ at~$\Delta$.
Suppose by contradiction that~$\calF = \emptyset$.
Then, for every~${\Delta \subseteq V_t}$, every cycle in~$\cC_{\Delta}$ crosses~$V_t$.
Let~$\cC = \{\cC_{\Delta}: \Delta \subseteq V_t\}$.

\begin{claim}\label{remark:lct-tw3-at-least-one-does-not-jump-Delta}
For every triple of vertices~${\Delta \subseteq V_t}$, at least
one cycle in~${\cC \setminus \cC_{\Delta}}$
does not jump~$\Delta$.
\end{claim}
\begin{proof}[Proof of Claim \ref{remark:lct-tw3-at-least-one-does-not-jump-Delta}]
Suppose by contradiction that all cycles
in~${\cC \setminus \cC_{abc}}$ jump~$abc$.
(The other cases are analogous.)
Then, as~${\cC_{abd},\cC_{bcd},\cC_{acd} \neq \emptyset}$,
for every~${\{i,j\} \subseteq abc}$,
there exists a longest cycle that 2-jumps~$abc$ at~$\{i,j\}$.
Hence, by Corollary~\ref{corollary:lct-tw3-lct>1-implies-not-jump} applied to~$abc$,
there exists a longest cycle~$C$ that
does not jump~$abc$.
If~$C$ intersects~$V_t$ at most once or is as item~$(ii)$ or~$(iii)$ of
Corollary~\ref{corollary:lct-tw3-lct>1-implies-not-jump},
then, as~$\twocross = \emptyset$,
$C$ is fenced by~$abc$, which implies 
that~$\calF \neq \emptyset$, a contradiction.
Thus, we may assume that~$C$ is as
item~$(i)$ of Corollary~\ref{corollary:lct-tw3-lct>1-implies-not-jump}, that
is, $C$ is outside~$abc$.

The rest of the proof is divided on two cases, whether $C$ 3-intersects or 4-intersects $V_t$. 

\setcounter{case}{0}
\begin{case}
$C$ {3-intersects}~$V_t$.

As all cycles
in~${\cC \setminus \cC_{abc}}$ jump~$abc$,
we have~${C \in \cC_{abc}}$.
Let~${D \in \cC_{abd}}$.
We know that~$D$ jumps~$abc$. If~$D$ jumps both~$acd$ and~$bcd$,
~${(D-D_{ab}) \cdot C_{ab}} \in \cC_{abd}$
but does not jump~$abc$, a contradiction (Figure~\ref{fig:lct=1-or-fenced-1}(a)).
If~$D$ jumps neither~$acd$ nor~$bcd$, 
then
~${C_{ab} \cdot D_{ab}}$ is a longest cycle that 2-crosses~$V_t$,
contradicting the fact that~$\twocross = \emptyset$
(Figure~\ref{fig:lct=1-or-fenced-1}(b)).
Hence, 
$\mbox{every }D \in \cC_{abd} \mbox{ jumps exactly one of } \{acd,bcd\}.$
Analogously,
we can conclude that 
$\mbox{every }F \in \cC_{acd} \mbox{ jumps exactly one of }\{bcd,abd\},$
and that
$\mbox{every }J \in \cC_{bcd} \mbox{ jumps exactly one of } \{abd,acd\}.$

Let~${(D,F,J) \in \cC_{abd} \times \cC_{acd} \times \cC_{bcd}}$.
Suppose that~$D$ jumps~$acd$ and~$F$ jumps~$abd$.
Then, 
both~${(D-D_{ab}) \cdot C_{bc} \cdot F_{ca}}$
and~${(F-F_{ac}) \cdot C_{cb} \cdot D_{ba}}$
are cycles, one of them longer than~$L$, a contradiction
(Figure~\ref{fig:lct=1-or-fenced-1}(c)).
Repeating the same argument with~$\{D,J\}$ and~$\{F,J\}$, 
we conclude that,
without loss of generality,
$D \mbox{ jumps } acd, F \mbox{ jumps } bcd, \mbox{ and } J \mbox{ jumps } abd.$  
But then,~${(D-D_{ab}) \cdot F_{ac} \cdot C_{cb}}$,
${(F-F_{ac}) \cdot J_{cb} \cdot C_{ba}}$,
and
${(J-J_{bc}) \cdot D_{ba} \cdot C_{ac}}$ are cycles
(Figure~\ref{fig:lct=1-or-fenced-1}(d)), a contradiction.
\end{case}


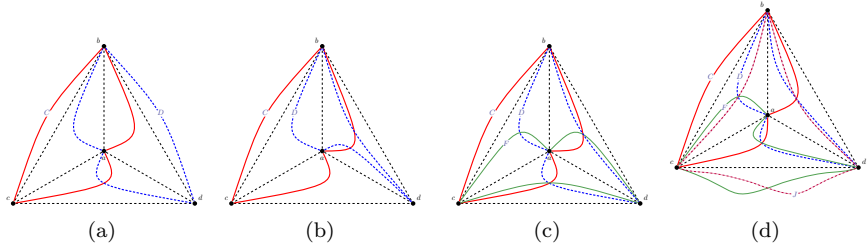
\begin{figure}[H]
\subfigure[ ]{
\resizebox{.22\textwidth}{!}{\input{Figures/lct=1-or-fenced-1}}
}
\subfigure[ ]{
\resizebox{.22\textwidth}{!}{\input{Figures/lct=1-or-fenced-2}}
}
\subfigure[ ]{
\resizebox{.22\textwidth}{!}{\input{Figures/lct=1-or-fenced-4}}
}
\subfigure[ ]{
\resizebox{.22\textwidth}{!}{\input{Figures/lct=1-or-fenced-5}}
}
\caption{Case 1 in the proof of
Claim \ref{remark:lct-tw3-at-least-one-does-not-jump-Delta}
of Lemma~\ref{lemma:lct-tw3-X2-empty-then-F}.
(a)~$D$ jumps both~$acd$ and~$bcd$.
(b)~$D$ does not jump neither~$acd$ nor~$bcd$.
(c)~$D$ jumps~$acd$ and~$F$ jumps~$abd$.
(d)~$D$ jumps~$acd$ and~$F$ jumps~$bcd$ and~$J$ jumps~$abd$.}
\label{fig:lct=1-or-fenced-1}
\end{figure}

\begin{case}
$C$ {4-intersects}~$V_t$.

Without loss of generality, suppose that~$C$ 
has~$abdc$-parts.
As~$C$ is~outside~$abc$, we have that
$C_{ab}$ is inside~$abd$
and~$C_{ac}$ is inside~$acd$.
By similar techniques to the previous cases, we can show
that
$\mbox{every cycle in } \cC_{abd} \mbox{ does not jump } acd.$
({Figures~\ref{fig:remark-at-least-one-does-not-jump-Delta-4inters}(a) and (b)}).
And, by symmetry, we can also show that
$\mbox{every cycle in } \cC_{acd} \mbox{ does not jump } abd.$
Let~$(D,F) \in \cC_{abd} \times \cC_{acd}$.
By 
the previous discussion,
$D_{ad}$ is outside~$acd$, and
$F_{ad}$ is outside~$abd$.
This implies that
$(D-D_{ad}) \cdot F_{ad} \in \cC_{abd}$
and jumps~$acd$, a contradiction 
(Figure \ref{fig:remark-at-least-one-does-not-jump-Delta-4inters}(c)).
\end{case}
This concludes the proof of the Claim.
\end{proof}

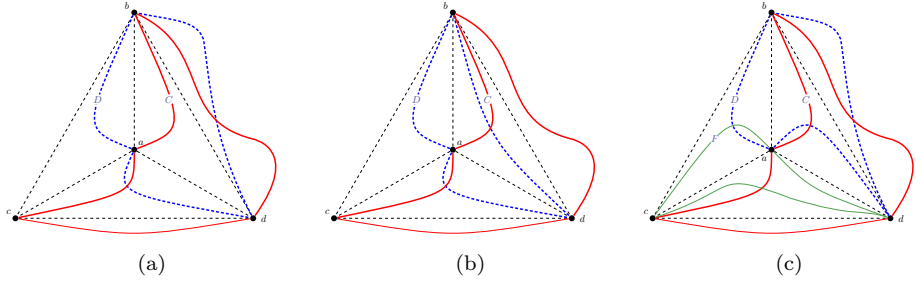
\begin{figure}[hbt]
\subfigure[ ]{
\resizebox{.33\textwidth}{!}{\input{Figures/lct-remark_one-outside-4-inters-1}}
}
\subfigure[ ]{
\resizebox{.33\textwidth}{!}{\input{Figures/lct-remark_one-outside-4-inters-2}}
}
\subfigure[ ]{
\resizebox{.33\textwidth}{!}{\input{Figures/lct-remark_one-outside-4-inters-3}}
}
\caption{Case 2 in the proof of Claim~\ref{remark:lct-tw3-at-least-one-does-not-jump-Delta}.
(a)~$D$ jumps~$acd$ and~$D_{bd}$ is~outside~$abd$.
In this situation,~${(D-D_{ab}) \cdot C_{ab} \in \cC_{abd}}$
but does not jump~$abc$.
(b)~$D$ jumps~$acd$ and~$D_{bd}$ is~inside~$abd$.
In this situation, ${C_{ab} \cdot C_{bd} \cdot D_{da} \in \cC_{abd}}$
but does not jump~$abc$.
(c)~$D$ does not jump~$acd$
 and~$F$ does not jump~$abd$.}
\label{fig:remark-at-least-one-does-not-jump-Delta-4inters}
\end{figure}

The next claim is used in the proofs
of Claims~\ref{remark:lct-tw3-remark-1-3crosses} and~\ref{remark:lct-tw3-remark-2-3crosses}.
It uses Claim~\ref{remark:lct-tw3-at-least-one-does-not-jump-Delta}.

\begin{claim}\label{remark:lct-tw3-remark-0-3crosses}
For every triple of vertices~${\Delta \subseteq V_t}$, every cycle
in~$\cC_{\Delta}$~jumps at least two of
${\{abd, acd, bcd, abc\} \setminus \{\Delta\}}$.
\end{claim}
\begin{proof}
[Proof of Claim~\ref{remark:lct-tw3-remark-0-3crosses}]
Without loss of generality, suppose that~$\Delta=abc$ and there exists a cycle, say~$C \in \cC_{abc}$, that
jumps at most one of~$\{abd,acd,bcd\}$. 
If~$C$ jumps no one of~$\{abd,acd,bcd\}$,
then~$C$ is inside~$abc$.
By Claim~\ref{remark:lct-tw3-at-least-one-does-not-jump-Delta},
there exists a longest cycle~$D \in \cC \setminus \cC_{abc}$ that does not jump~$abc$.
By joining two parts of $C$ and $D$, we obtain a cycle
in $\twocross$, a contradiction.

So,
without loss of generality suppose $C$ jumps~$abd$.
Let~$D \in \cC_{abd}$.
If~$D$ jumps~$abc$,
then, ${C_{ab} \cdot D_{ab}} \in \twocross$,
a contradiction.
(Figure~\ref{fig:lct-remark-0-3crosses}(a)).
Hence,~$D$ is outside~$abc$.
Thus, as both~${(C-C_{ab}) \cdot D_{ab}}$
and~${(C-C_{ab}) \cdot (D-D_{ab})}$ are cycles, we
conclude that
$|C_{ab}| \geq L/2.$

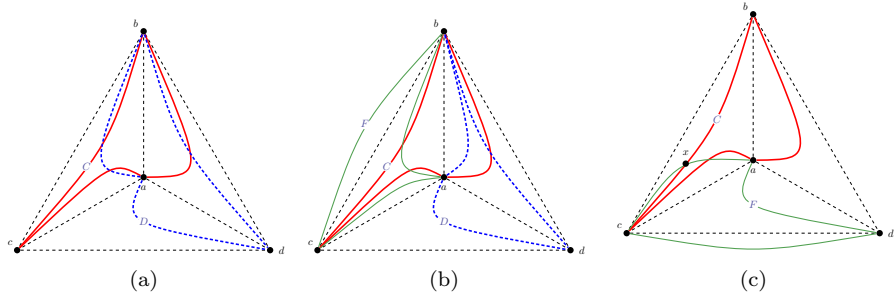
\begin{figure}[H]
\center
\subfigure[ ]{
\resizebox{.31\textwidth}{!}{\input{Figures/lct-remark-0-3crosses-1}}
}
\subfigure[ ]{
\resizebox{.31\textwidth}{!}{\input{Figures/lct-remark-0-3crosses-2}}
}
\subfigure[ ]{
\resizebox{.31\textwidth}{!}{\input{Figures/lct-remark-0-3crosses-3}}
}
\caption{Cases in the proof of Claim~\ref{remark:lct-tw3-remark-0-3crosses}
of Lemma~\ref{lemma:lct-tw3-3crosses}.
(a)~$D$ jumps~$abc$.
(b)~$D$ is outside~$abc$ and~$F \in \cC_{abc}$.
(c)~$D$ is outside~$abc$ and~$F \in \cC_{acd}$.
}
\label{fig:lct-remark-0-3crosses}
\end{figure}

By Claim \ref{remark:lct-tw3-at-least-one-does-not-jump-Delta},
there exists a longest cycle~$F \in \cC \setminus \cC_{abd}$
that is~outside~$abd$.
If~${F \in \cC_{abc}}$ then 
${C_{ab} \cdot F_{ab}} \in \twocross$, a contradiction.
(Figure~\ref{fig:lct-remark-0-3crosses}(b)).
If~$F \in \cC_{acd}$ then
there exists a subpath~$R$ of~$C_{bc}$ starting at $b$, that
ends at a vertex~$x$ of~$F_{ac}$ and is internally disjoint
from~$F$.
Then, 
both~${C_{ab} \cdot R \cdot F_{xa}}$
and~${C_{ab} \cdot R \cdot (F-F_{xa})}$
are cycles, a contradiction
(Figure~\ref{fig:lct-remark-0-3crosses}(c)).

The case in which~$F \in \cC_{bcd}$ is symmetric
to the case in which $F \in \cC_{acd}$.
\end{proof}

\begin{claim}\label{remark:lct-tw3-remark-1-3crosses}
There exists a triple of vertices~${\Delta \subseteq V_t}$ such
that every cycle in~$\cC_{\Delta}$ jumps exactly two
of~${\{abd,acd,bcd,abc\} \setminus \{\Delta\}}$.
\end{claim}
\begin{proof}[Proof of Claim~\ref{remark:lct-tw3-remark-1-3crosses}]
By Claim \ref{remark:lct-tw3-remark-0-3crosses}, for 
each triple of vertices~$\Delta \subseteq V_t$,
every cycle in~$\cC_{\Delta}$
jumps two or three of~${\{abd, acd, bcd, abc\} \setminus \{\Delta\}}$.
Suppose by contradiction that for each~$\Delta \subseteq V_t$ 
there exists
a longest cycle
in~$\cC_{\Delta}$ that
jumps the three of~${\{abd, acd, bcd, abc\} \setminus \{\Delta\}}$.
Let~$C$,~$D$, and~$F$  be corresponding cycles in~{$\cC_{abc}$,
$\cC_{abd}$}, and~$\cC_{bcd}$ respectively.
We can easily conclude that
$|F_{bc}|=|C_{bc}|$
and
$|D_{bd}|=|F_{bd}|.$

\begin{figure}[hbt]
\center
\resizebox{.45\textwidth}{!}{\input{Figures/lct-remark-1-3crosses}}
\caption{Situation in the proof of Claim~\ref{remark:lct-tw3-remark-1-3crosses}
of Lemma~\ref{lemma:lct-tw3-3crosses}.}
\label{fig:lct-remark-1-3crosses}
\end{figure}
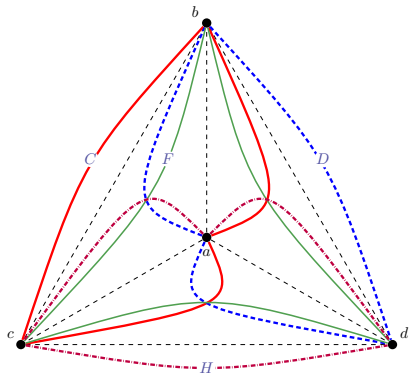

Let~$H$ be a corresponding cycle in~$\cC_{acd}$.
We have that both~${F_{cb}\cdot C_{ba}\cdot D_{ad}\cdot H_{dc}}$
and~${F_{db}\cdot D_{ba}\cdot C_{ac}\cdot H_{cd}}$
are cycles, and their lengths sum more than
$|F_{bc}|+ |F_{bd}|+|C_{ab}|+|C_{ac}|+|D_{ab}|+|D_{ad}|$.
By 
the last part of the previous paragraph,
this sum equals
$|C_{bc}|+ |D_{bd}|+|C_{ab}|+|C_{ac}|+|D_{ab}|+|D_{ad}|=|C|+|D|=2L$,
a contradiction
(Figure~\ref{fig:lct-remark-1-3crosses}).
\end{proof}

Suppose that~$abc$ is the triangle given by Claim~\ref{remark:lct-tw3-remark-1-3crosses}.
Let~${C \in \cC_{abc}}$ 
We may assume without loss of generality that~$C$ jumps~$abd$ and~$acd$,
and is outside~$bcd$.
By Claim~\ref{remark:lct-tw3-at-least-one-does-not-jump-Delta},
there exist cycles~${F,H,J \in \cC}$
such that~$F$ {2-intersects}~$abd$ but does not jump~$abd$,
$H$ {2-intersects}~$acd$ but does not jump~$acd$, and
$J$ {2-intersects}~$abc$ but does not jump~$abc$.


By similar arguments to the previous Claims, we
can show the next claim.

\begin{claim}\label{remark:lct-tw3-remark-2-3crosses}
$F \in \cC_{bcd}$ and~$H \in \cC_{bcd}$ \mbox{(Figure~\ref{fig:lct-remark-2-3crosses-1}).}
\end{claim}

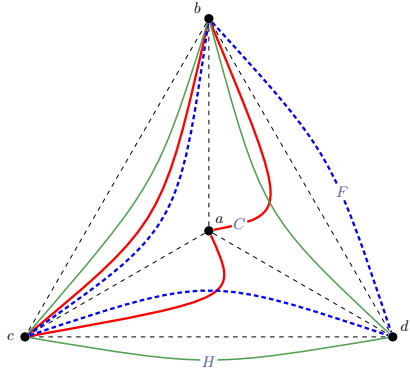
\begin{figure}[!hbt]
\center
\resizebox{.45\textwidth}{!}{\input{Figures/lct-remark-2-3crosses-4}}
\caption{Situation of Claim~\ref{remark:lct-tw3-remark-2-3crosses}
of Lemma~\ref{lemma:lct-tw3-3crosses}.}
\label{fig:lct-remark-2-3crosses-1}
\end{figure}

The rest of the proof of Lemma \ref{lemma:lct-tw3-X2-empty-then-F}
is divided on two cases depending on~$J$.
If $J \in \cC_{bcd}$ (Figure~\ref{fig:3crosses-J}(a)),
then
${(C-C_{bc}) \cdot J_{bc}} \in \cC_{abc}$
and jumps the three of~$\{abd,acd,bcd\}$,
a contradiction to the choice of~$abc$
as in Claim \ref{remark:lct-tw3-remark-1-3crosses}.
Thus, 
${J \in \cC_{abd}}$ (Figure~\ref{fig:3crosses-J}(b)).

By Claim~\ref{remark:lct-tw3-remark-0-3crosses},~$J$ jumps both~$acd$ and~$bcd$.
By joining parts of cycles as in the previous claims, we can show that
$|F_{bd}| \geq L/2.$
Now, let~$R$ be the subpath of~$F_{cd}$ that
is internally disjoint from~$C_{ac}$,
starts at~$d$ and ends at  
a vertex~$x$ in~$C_{ac}$.
Let~$S$ be the subpath of~$C$ that starts at~$x$, ends at~$b$, and
is such that~$|S|\geq L/2$.
Then, 
~${F_{bd} \cdot R \cdot S}$ is a cycle longer
than~$L$, a contradiction.
This concludes the proof of the lemma.
\end{proof}

\begin{figure}[ht]
\subfigure[ ]{
\resizebox{.45\textwidth}{!}{\input{Figures/lct-3crosses-J-case2}}
}
\subfigure[ ]{
\resizebox{.45\textwidth}{!}{\input{Figures/lct-3crosses-J-case1}}
}
\caption{Cases in the last part of the proof of Lemma~\ref{lemma:lct-tw3-3crosses}.
(a)~$J \in \cB_{bcd}$ (Case 1). 
(b)~$J \in \cB_{abd}$ (Case 2).}
\label{fig:3crosses-J}
\end{figure}
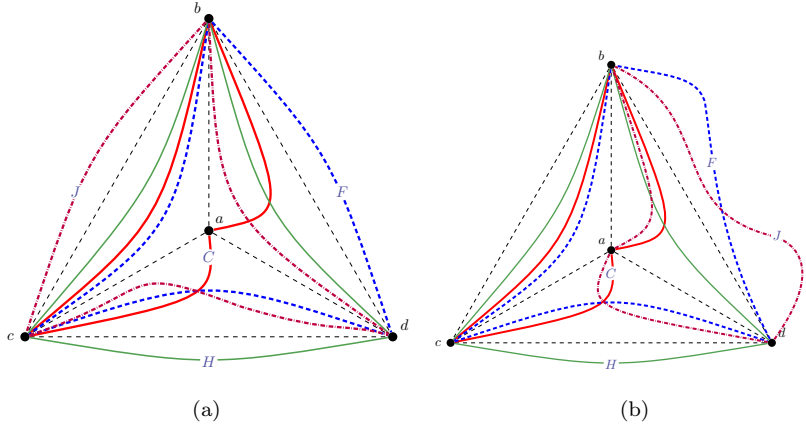

%% file: Figures/lct-2-crosses-are-equivalent-1.tex
\begin{tikzpicture}
[scale=0.7, label distance=3pt, every node/.style={fill,circle,inner sep=0pt,minimum size=6pt}]
 
    \node at (0,0)[myblue,label=above left :$c$,fill=black, circle](c) {};
    \node at (12,0)[myblue,label=above right:$d$,fill=black, circle](d) {};
    \node at (6,10.3923)[myblue,label=above left:$b$,fill=black, circle](b) {};
    \node at (6,3.4641)[myblue,label=below:$a$,fill=black, circle](a) {};
    
    \draw [dashed] (a) -- (b);
    \draw [dashed] (b) -- (c);
    \draw [dashed] (a) -- (c);
    \draw [dashed] (a) -- (d);
    \draw [dashed] (b) -- (d);
    \draw [dashed] (c) -- (d);
    
    \draw [densely dashed] [color=blue] [line width=1.5pt] (a) .. controls (3, 4) .. (b);
     \draw [densely dashed][color=blue] [line width=1.5pt] (a) .. controls (4.5, 5.5) .. (b);
     
    \node at (4,6)[myblue,label=below:$ $,fill=white, circle](C) {$C$};
    
     \draw [color=mygreen] [line width=1.5pt] (b) .. controls (5, 4.5) .. (c);
     \draw [color=mygreen] [line width=1.5pt] (b) .. controls (4.3, 5) .. (c);
    
     \draw [color=red] [line width=1pt] (a) .. controls (4, 5.5) .. (c);
     \draw [color=red] [line width=1pt] (a) .. controls (3, 3.25) .. (c);
    
    \node at (2.85,4)[myblue,label=below:$ $,fill=white, circle](D) {$D$};
    
     \draw [loosely dashdotted][color=black] [line width=1pt] (b) .. controls (6.5, 5.5) and (8,7) .. (8,4) .. controls(8.3,2.3) and (10.5,2) ..(d);
    
     \draw [loosely dashdotted][color=black] [line width=1pt] (c) .. controls (3.5, 1) and (3,2) .. (6,1.4) .. controls(8.3,0.5) and (7.3,2) ..(d);
     
    \node at (7.5,6)[myblue,label=below:$ $,fill=white, circle](C) {$P$};
    
    \node at (5.5,1.4)[myblue,label=below:$ $,fill=white, circle](C) {$Q$};

  \end{tikzpicture}

%% file: Figures/lct-2-crosses-are-equivalent-2.tex
\begin{tikzpicture}
[scale=0.7, label distance=3pt, every node/.style={fill,circle,inner sep=0pt,minimum size=6pt}]
 
    \node at (0,0)[myblue,label=above left :$c$,fill=black, circle](c) {};
    \node at (12,0)[myblue,label=above right:$d$,fill=black, circle](d) {};
    \node at (6,10.3923)[myblue,label=above left:$b$,fill=black, circle](b) {};
    \node at (6,3.4641)[myblue,label=below:$a$,fill=black, circle](a) {};
    
    \draw [dashed] (a) -- (b);
    \draw [dashed] (b) -- (c);
    \draw [dashed] (a) -- (c);
    \draw [dashed] (a) -- (d);
    \draw [dashed] (b) -- (d);
    \draw [dashed] (c) -- (d);
    
    \draw [densely dashed] [color=blue] [line width=1.5pt] (a) .. controls (3.35, 4.5) .. (b);
     \draw [densely dashed][color=blue] [line width=1.5pt] (a) .. controls (8.65, 4.5) .. (b);
     
    \node at (7.75,6)[myblue,label=below:$ $,fill=white, circle](C) {$C$};
    
     \draw [color=mygreen] [line width=1.5pt] (b) .. controls (5, 3.5) .. (c);
     \draw [color=mygreen] [line width=1.5pt] (b) .. controls (4.3, 5) .. (c);
    
     \draw [color=red] [line width=1pt] (a) .. controls (4, 6) .. (c);
     \draw [color=red] [line width=1pt] (a) .. controls (7, 1.25) .. (c);
    
    \node at (2.85,4)[myblue,label=below:$ $,fill=white, circle](D) {$D$};
    
    \node at (4.25,3.6)[myblue,label=below:$ $,fill=white, circle](F) {$F$};

    
     
    
  \end{tikzpicture}

%% file: Figures/lct-2-crosses-are-equivalent-3.tex
  \begin{tikzpicture}
[scale=0.7, label distance=3pt, every node/.style={fill,circle,inner sep=0pt,minimum size=6pt}]
 
    \node at (0,0)[myblue,label=above left :$c$,fill=black, circle](c) {};
    \node at (12,0)[myblue,label=above right:$d$,fill=black, circle](d) {};
    \node at (6,10.3923)[myblue,label=above left:$b$,fill=black, circle](b) {};
    \node at (6,3.4641)[myblue,label=below:$a$,fill=black, circle](a) {};
    
    \draw [dashed] (a) -- (b);
    \draw [dashed] (b) -- (c);
    \draw [dashed] (a) -- (c);
    \draw [dashed] (a) -- (d);
    \draw [dashed] (b) -- (d);
    \draw [dashed] (c) -- (d);
    
    \draw [densely dashed] [color=blue] [line width=1.5pt] (a) .. controls (3.35, 4.5) .. (b);
     \draw [densely dashed][color=blue] [line width=1.5pt] (a) .. controls (8.65, 4.5) .. (b);
     
    \node at (7.75,6)[myblue,label=below:$ $,fill=white, circle](C) {$C$};
    
     \draw [color=mygreen] [line width=1.5pt] (b) .. controls (4.7, 5) .. (c);
     \draw [color=mygreen] [line width=1.5pt] (b) .. controls (7.3, 5) .. (d);
     \draw [color=mygreen] [line width=1.5pt] (c) .. controls (6, 2) .. (d);

     \draw [color=red] [line width=1pt] (a) .. controls (4, 5.45) .. (c);
     \draw [color=red] [line width=1pt] (a) .. controls (7, 1.25) .. (c);
    
    \node at (2.85,4)[myblue,label=below:$ $,fill=white, circle](D) {$D$};
    
    \node at (8.6,3.6)[myblue,label=below:$ $,fill=white, circle](F) {$F$};
    
      \node at (7.95,4.65)[myblue,label=below:$x$,fill=black, circle](x) {};
      \node at (6.25,1.5)[myblue,label=below:$y$,fill=black, circle](y) {};

    
     
    
  \end{tikzpicture}

%% file: Figures/lct-2-crosses-are-equivalent-4.tex
\begin{tikzpicture}
[scale=0.7, label distance=3pt, every node/.style={fill,circle,inner sep=0pt,minimum size=6pt}]
 
    \node at (0,0)[myblue,label=above left :$c$,fill=black, circle](c) {};
    \node at (12,0)[myblue,label=above right:$d$,fill=black, circle](d) {};
    \node at (6,10.3923)[myblue,label=above left:$b$,fill=black, circle](b) {};
    \node at (6,3.4641)[myblue,label=below:$a$,fill=black, circle](a) {};
    
    \draw [dashed] (a) -- (b);
    \draw [dashed] (b) -- (c);
    \draw [dashed] (a) -- (c);
    \draw [dashed] (a) -- (d);
    \draw [dashed] (b) -- (d);
    \draw [dashed] (c) -- (d);
    
    \draw [densely dashed] [color=blue] [line width=1.5pt] (a) .. controls (3.35, 4.5) .. (b);
     \draw [densely dashed][color=blue] [line width=1.5pt] (a) .. controls (8.65, 4.5) .. (b);
     
    \node at (7.75,6)[myblue,label=below:$ $,fill=white, circle](C) {$C$};
    
     \draw [color=mygreen] [line width=1.5pt] (b) .. controls (4.7, 5) .. (c);
     \draw [color=mygreen] [line width=1.5pt] (b) .. controls (7.3, 5) .. (d);
     \draw [color=mygreen] [line width=1.5pt] (c) .. controls (6, 2) .. (d);

     \draw [color=red] [line width=1pt] (a) .. controls (4, 5.45) .. (c);
     \draw [color=red] [line width=1pt] (a) .. controls (4, 3.45) .. (c);
    
    \node at (2.85,4)[myblue,label=below:$ $,fill=white, circle](D) {$D$};
    
    \node at (8.6,3.6)[myblue,label=below:$ $,fill=white, circle](F) {$F$};
    
      \node at (7.95,4.65)[myblue,label=below:$x$,fill=black, circle](x) {};
    
    
     
    
  \end{tikzpicture}

%% file: Figures/lct-2-crosses-implies-fenced-1.tex
  \begin{tikzpicture}
[scale=0.7, label distance=3pt, every node/.style={fill,circle,inner sep=0pt,minimum size=6pt}]
 
    \node at (0,0)[myblue,label=above left :$c$,fill=black, circle](c) {};
    \node at (12,0)[myblue,label=above right:$d$,fill=black, circle](d) {};
    \node at (6,10.3923)[myblue,label=above left:$b$,fill=black, circle](b) {};
    \node at (6,3.4641)[myblue,label=below:$a$,fill=black, circle](a) {};
    
    \draw [dashed] (a) -- (b);
    \draw [dashed] (b) -- (c);
    \draw [dashed] (a) -- (c);
    \draw [dashed] (a) -- (d);
    \draw [dashed] (b) -- (d);
    \draw [dashed] (c) -- (d);
    
    \draw [densely dashed] [color=blue] [line width=1.5pt] (a) .. controls (3.35, 4.5) .. (b);
     \draw [densely dashed][color=blue] [line width=1.5pt] (a) .. controls (8.65, 4.5) .. (b);
     
    \node at (7.75,6)[myblue,label=below:$ $,fill=white, circle](C) {$C''$};
    \node at (4.25,6)[myblue,label=below:$ $,fill=white, circle](C) {$C'$};
    
     \draw [color=red] [line width=1.5pt] (b) .. controls (4.7, 5) .. (c);
     \draw [color=red] [line width=1.5pt] (b) .. controls (10, 6) .. (d);

    \draw [color=mygreen] [line width=1pt] (a) .. controls (4, 5.45) .. (c);
     \draw [color=mygreen] [line width=1pt] (a) .. controls (8, 5.45) .. (d);
    
    \node at (2.85,4)[myblue,label=below:$ $,fill=white, circle](D) {$F$};
    
    \node at (10,5.4)[myblue,label=below:$ $,fill=white, circle](F) {$D$};
    
      \node at (7.95,4.65)[myblue,label=below:$x$,fill=black, circle](x) {};
   
     \draw [color=mygreen] [line width=1pt] (c) .. controls (4.5, 1.5) and (3,2.5) .. (7,1) .. controls(10.3,0) and (10.5,0.5) ..(d);
     
     \draw [color=red] [line width=1pt] (d) .. controls (7.5, 1.5) and (9,2.5) .. (5,1) .. controls(1.7,0) and (1.5,0.5) ..(c);

    
     
    
  \end{tikzpicture}

%% file: Figures/lct-2-crosses-implies-fenced-2.tex
  \begin{tikzpicture}
[scale=0.7, label distance=3pt, every node/.style={fill,circle,inner sep=0pt,minimum size=6pt}]
 
    \node at (0,0)[myblue,label=above left :$c$,fill=black, circle](c) {};
    \node at (12,0)[myblue,label=above right:$d$,fill=black, circle](d) {};
    \node at (6,10.3923)[myblue,label=above left:$b$,fill=black, circle](b) {};
    \node at (6,3.4641)[myblue,label=below:$a$,fill=black, circle](a) {};
    
    \draw [dashed] (a) -- (b);
    \draw [dashed] (b) -- (c);
    \draw [dashed] (a) -- (c);
    \draw [dashed] (a) -- (d);
    \draw [dashed] (b) -- (d);
    \draw [dashed] (c) -- (d);
    
    \draw [densely dashed] [color=blue] [line width=1.5pt] (a) .. controls (3.35, 4.5) .. (b);
     \draw [densely dashed][color=blue] [line width=1.5pt] (a) .. controls (8.65, 4.5) .. (b);
     
    \node at (7.75,6)[myblue,label=below:$ $,fill=white, circle](C) {$C''$};
    \node at (4.25,6)[myblue,label=below:$ $,fill=white, circle](C) {$C'$};
    
     \draw [color=red] [line width=1.5pt] (b) .. controls (4.7, 5) .. (c);
     \draw [color=red] [line width=1.5pt] (b) .. controls (10, 6) .. (d);
     \draw [color=mygreen] [line width=1pt] (c) .. controls (6, -1) .. (d);

     \draw [color=mygreen] [line width=1pt] (a) .. controls (8, 5.45) .. (d);
     \draw [color=mygreen] [line width=1pt] (a) .. controls (7, 1.25) .. (c);
    
    \node at (3.75,0.75)[myblue,label=below:$ $,fill=white, circle](D) {$F$};
    
    \node at (10,5.4)[myblue,label=below:$ $,fill=white, circle](F) {$D$};
    
   
     \draw [color=red] [line width=1pt] (c) .. controls (4.5, 1.5) and (3,2.5) .. (7,1) .. controls(10.3,0) and (10.5,0.5) ..(d);

    
     
    
  \end{tikzpicture}

%% file: Figures/lct-2-crosses-implies-fenced-3.tex
  \begin{tikzpicture}
[scale=0.7, label distance=3pt, every node/.style={fill,circle,inner sep=0pt,minimum size=6pt}]
 
    \node at (0,0)[myblue,label=above left :$c$,fill=black, circle](c) {};
    \node at (12,0)[myblue,label=above right:$d$,fill=black, circle](d) {};
    \node at (6,10.3923)[myblue,label=above left:$b$,fill=black, circle](b) {};
    \node at (6,3.4641)[myblue,label=below:$a$,fill=black, circle](a) {};
    
    \draw [dashed] (a) -- (b);
    \draw [dashed] (b) -- (c);
    \draw [dashed] (a) -- (c);
    \draw [dashed] (a) -- (d);
    \draw [dashed] (b) -- (d);
    \draw [dashed] (c) -- (d);
    
    \draw [densely dashed] [color=blue] [line width=1.5pt] (a) .. controls (3.35, 4.5) .. (b);
     \draw [densely dashed][color=blue] [line width=1.5pt] (a) .. controls (8.65, 4.5) .. (b);
     
    \node at (7.75,6)[myblue,label=below:$ $,fill=white, circle](C) {$C''$};
    \node at (4.25,6)[myblue,label=below:$ $,fill=white, circle](C) {$C'$};
    
     \draw [color=red] [line width=1.5pt] (b) .. controls (4.7, 5) .. (c);
     \draw [color=red] [line width=1.5pt] (b) .. controls (10, 6) .. (d);

     \draw [color=mygreen] [line width=1pt] (a) .. controls (8, 5.45) .. (d);
     \draw [color=mygreen] [line width=1pt] (a) .. controls (4, 1.25) .. (c);

    \node at (10,5.4)[myblue,label=below:$ $,fill=white, circle](F) {$D$};
    
      \node at (4.05,4.65)[myblue,label=below:$x$,fill=black, circle](x) {};
   
     \draw [color=red] [line width=1pt] (c) .. controls (4.5, 1.5) and (3,2.5) .. (7,1) .. controls(10.3,0) and (10.5,0.5) ..(d);
     
     \draw [color=mygreen] [line width=1pt] (d) .. controls (7.5, 1.5) and (9,2.5) .. (5,1) .. controls(1.7,0) and (1.5,0.5) ..(c);
     
    \node at (8.75,4.25)[myblue,label=below:$ $,fill=white, circle](D) {$F$};
    
    
     
    
  \end{tikzpicture}

%% file: Figures/lct-2-crosses-implies-fenced-4.tex
  \begin{tikzpicture}
[scale=0.7, label distance=3pt, every node/.style={fill,circle,inner sep=0pt,minimum size=6pt}]
 
    \node at (0,0)[myblue,label=above left :$c$,fill=black, circle](c) {};
    \node at (12,0)[myblue,label=above right:$d$,fill=black, circle](d) {};
    \node at (6,10.3923)[myblue,label=above left:$b$,fill=black, circle](b) {};
    \node at (6,3.4641)[myblue,label=below:$a$,fill=black, circle](a) {};
    
    \draw [dashed] (a) -- (b);
    \draw [dashed] (b) -- (c);
    \draw [dashed] (a) -- (c);
    \draw [dashed] (a) -- (d);
    \draw [dashed] (b) -- (d);
    \draw [dashed] (c) -- (d);
    
    \draw [densely dashed] [color=blue] [line width=1.5pt] (a) .. controls (3.35, 4.5) .. (b);
     \draw [densely dashed][color=blue] [line width=1.5pt] (a) .. controls (8.65, 4.5) .. (b);
     
    \node at (7.75,6)[myblue,label=below:$ $,fill=white, circle](C) {$C''$};
    \node at (4.25,6)[myblue,label=below:$ $,fill=white, circle](C) {$C'$};
    
     \draw [color=red] [line width=1.5pt] (b) .. controls (4.7, 5) .. (c);
     \draw [color=red] [line width=1.5pt] (b) .. controls (10, 6) .. (d);

    \draw [color=mygreen] [line width=1pt] (a) .. controls (4, 5.45) .. (c);
     \draw [color=mygreen] [line width=1pt] (a) .. controls (8, 1.45) .. (d);
    
    \node at (2.85,4)[myblue,label=below:$ $,fill=white, circle](D) {$F$};
    
    \node at (10,5.4)[myblue,label=below:$ $,fill=white, circle](F) {$D$};
    
   
     \draw [color=mygreen] [line width=1pt] (c) .. controls (4.5, 1.5) and (3,2.5) .. (7,1) .. controls(10.3,0) and (10.5,0.5) ..(d);
     
     \draw [color=red] [line width=1pt] (d) .. controls (7.5, 1.5) and (9,2.5) .. (5,1) .. controls(1.7,0) and (1.5,0.5) ..(c);

    
     
    
  \end{tikzpicture}

%% file: Figures/lct-remark-1.tex
  \begin{tikzpicture}
[scale=0.7, label distance=3pt, every node/.style={fill,circle,inner sep=0pt,minimum size=6pt}]
 
    \node at (0,0)[myblue,label=above left :$c$,fill=black, circle](c) {};
    \node at (12,0)[myblue,label=above right:$d$,fill=black, circle](d) {};
    \node at (6,10.3923)[myblue,label=above left:$b$,fill=black, circle](b) {};
    \node at (6,3.4641)[myblue,label=below:$a$,fill=black, circle](a) {};
    
    \draw [dashed] (a) -- (b);
    \draw [dashed] (b) -- (c);
    \draw [dashed] (a) -- (c);
    \draw [dashed] (a) -- (d);
    \draw [dashed] (b) -- (d);
    \draw [dashed] (c) -- (d);
    
    \draw [densely dashed] [color=blue] [line width=1.5pt] (a) .. controls (3.35, 4.5) .. (b);
     \draw [densely dashed][color=blue] [line width=1.5pt] (a) .. controls (8.65, 4.5) .. (b);
     
    \node at (7.75,6)[myblue,label=below:$ $,fill=white, circle](C) {$C''$};
    \node at (4.25,6)[myblue,label=below:$ $,fill=white, circle](C) {$C'$};
    
     \draw [color=red] [line width=1.5pt] (b) .. controls (4.7, 5) .. (c);
     \draw [color=red] [line width=1.5pt] (b) .. controls (7.3, 5) .. (d);

    \draw [color=mygreen] [line width=1pt] (a) .. controls (4, 5.45) .. (c);
     \draw [color=mygreen] [line width=1pt] (a) .. controls (8, 5.45) .. (d);
    
    \node at (2.85,4)[myblue,label=below:$ $,fill=white, circle](D) {$F$};

   
     \draw [color=mygreen] [line width=1pt] (c) .. controls (4.5, 1.5) and (3,2.5) .. (7,1) .. controls(10.3,0) and (10.5,0.5) ..(d);
     
     \draw [color=red] [line width=1pt] (d) .. controls (7.5, 1.5) and (9,2.5) .. (5,1) .. controls(1.7,0) and (1.5,0.5) ..(c);

    \node at (8,1.75)[myblue,label=below:$ $,fill=white, circle](D) {$D$};
    
    
     
    
  \end{tikzpicture}

%% file: Figures/lct-2-crosses-implies-fenced-5.tex
  \begin{tikzpicture}
[scale=0.7, label distance=3pt, every node/.style={fill,circle,inner sep=0pt,minimum size=6pt}]
 
    \node at (0,0)[myblue,label=above left :$c$,fill=black, circle](c) {};
    \node at (12,0)[myblue,label=above right:$d$,fill=black, circle](d) {};
    \node at (6,10.3923)[myblue,label=above left:$b$,fill=black, circle](b) {};
    \node at (6,3.4641)[myblue,label=below:$a$,fill=black, circle](a) {};
    
    \draw [dashed] (a) -- (b);
    \draw [dashed] (b) -- (c);
    \draw [dashed] (a) -- (c);
    \draw [dashed] (a) -- (d);
    \draw [dashed] (b) -- (d);
    \draw [dashed] (c) -- (d);
    
    \draw [densely dashed] [color=blue] [line width=1.5pt] (a) .. controls (3.35, 4.5) .. (b);
    \draw [densely dashed][color=blue] [line width=1.5pt] (a) .. controls (8.65, 4.5) .. (b);
    \draw [color=mygreen] [line width=1pt] (a) .. controls (7.65, 5.5) .. (b);
     
    \node at (7.95,6)[myblue,label=below:$ $,fill=white, circle](C) {$C''$};
    \node at (4.25,6)[myblue,label=below:$ $,fill=white, circle](C) {$C'$};
    
     \draw [color=red] [line width=1.5pt] (b) .. controls (4.7, 5) .. (c);
     \draw [color=mygreen] [line width=1pt] (b) .. controls (10, 6) .. (d);
     \draw [color=red] [line width=1.5pt] (b) .. controls (7.3, 5) .. (d);
     \draw [color=mygreen] [line width=1pt] (c) .. controls (6, -1) .. (d);

     \draw [color=mygreen] [line width=1pt] (a) .. controls (7, 1.25) .. (c);
    
    
    \node at (10.15,5)[myblue,label=below:$ $,fill=white, circle](D) {$H$};

      \node at (6,1.35)[myblue,label=below:$x$,fill=black, circle](x) {};
   
     \draw [color=red] [line width=1pt] (c) .. controls (4.5, 1.5) and (3,2.5) .. (7,1) .. controls(10.3,0) and (10.5,0.5) ..(d);
     

    \node at (8,0.75)[myblue,label=below:$ $,fill=white, circle](D) {$D$};
    
    
     
    
    
    \node at (6,1.35)[myblue,label=below:$x$,fill=black, circle](x) {};
    
  \end{tikzpicture}

%% file: Figures/lct-2-crosses-implies-fenced-6.tex
  \begin{tikzpicture}
[scale=0.7, label distance=3pt, every node/.style={fill,circle,inner sep=0pt,minimum size=6pt}]
 
    \node at (0,0)[myblue,label=above left :$c$,fill=black, circle](c) {};
    \node at (12,0)[myblue,label=above right:$d$,fill=black, circle](d) {};
    \node at (6,10.3923)[myblue,label=above left:$b$,fill=black, circle](b) {};
    \node at (6,3.4641)[myblue,label=below:$a$,fill=black, circle](a) {};
    
    \draw [dashed] (a) -- (b);
    \draw [dashed] (b) -- (c);
    \draw [dashed] (a) -- (c);
    \draw [dashed] (a) -- (d);
    \draw [dashed] (b) -- (d);
    \draw [dashed] (c) -- (d);
    
     
    
     \draw [color=red] [line width=1.5pt] (b) .. controls (7.6, 5) .. (d);
      \draw [densely dashed][color=blue] [line width=1.5pt] (b) .. controls (4.4, 5) .. (c);

    \draw [color=mygreen] [line width=1pt] (a) .. controls (4, 5.45) .. (c);
     \draw [color=mygreen] [line width=1pt] (a) .. controls (8, 5.45) .. (d);
     \draw [color=red] [line width=1.5pt] (a) .. controls (7.5, 4.25) .. (d);
     \draw [densely dashed][color=blue] [line width=1.5pt] (a) .. controls (4.5, 4.25) .. (c);

   
     
     
     \draw [color=mygreen] [line width=1pt] (c) .. controls (6, 2) .. (d);
    \node at (6,1.5)[myblue,label=below:$ $,fill=white, circle](D) {$F$};
    
    \draw [color=red] [line width=1.5pt] (a) .. controls (7.5, 4.5) .. (b);
    
    \draw [densely dashed][color=blue] [line width=1.5pt] (a) .. controls (4.5, 4.5) .. (b);
    
    \node at (7.65,5.75)[myblue,label=below:$ $,fill=white, circle](H) {$H$};
    \node at (4.35,5.75)[myblue,label=below:$ $,fill=white, circle](J) {$J$};

    
     
    
  \end{tikzpicture}

%% file: Figures/lct=1-or-fenced-1.tex
  \begin{tikzpicture}
[scale=0.7, label distance=3pt, every node/.style={fill,circle,inner sep=0pt,minimum size=6pt}]
 
    \node at (6,3.4641)[myblue,label=below:$a$,fill=black, circle](a) {};
    \node at (6,10.3923)[myblue,label=above left:$b$,fill=black, circle](b) {};
    \node at (0,0)[myblue,label=above left :$c$,fill=black, circle](c) {};
    \node at (12,0)[myblue,label=above right:$d$,fill=black, circle](d) {};
    
    \draw [dashed] (a) -- (b);
    \draw [dashed] (b) -- (c);
    \draw [dashed] (a) -- (c);
    \draw [dashed] (a) -- (d);
    \draw [dashed] (b) -- (d);
    \draw [dashed] (c) -- (d);
    
    \draw [densely dashed] [color=blue] [line width=1.5pt] (a) .. controls (3.35, 4.5) .. (b);
    \draw [color=red] [line width=1.5pt] (a) .. controls (8.65, 4.5) .. (b);
    
    \draw [color=red] [line width=1.5pt] (a) .. controls (7, 1.25) .. (c);
    
    \draw[densely dashed] [color=blue] [line width=1.5pt]  (a) .. controls (5, 1.25) .. (d);

    \draw [color=red] [line width=1.5pt] (b) .. controls (2, 6) .. (c);

    \draw [densely dashed] [color=blue] [line width=1.5pt]  (b) .. controls (10, 6) .. (d);

    \node at (2.25,6)[myblue,label=below:$ $,fill=white, circle](C) {$C$};
    \node at (9.75,6)[myblue,label=below:$ $,fill=white, circle](D) {$D$};
    
  \end{tikzpicture}

%% file: Figures/lct=1-or-fenced-2.tex
  \begin{tikzpicture}
[scale=0.7, label distance=3pt, every node/.style={fill,circle,inner sep=0pt,minimum size=6pt}]
 
    \node at (6,3.4641)[myblue,label=below:$a$,fill=black, circle](a) {};
    \node at (6,10.3923)[myblue,label=above left:$b$,fill=black, circle](b) {};
    \node at (0,0)[myblue,label=above left :$c$,fill=black, circle](c) {};
    \node at (12,0)[myblue,label=above right:$d$,fill=black, circle](d) {};
    
    \draw [dashed] (a) -- (b);
    \draw [dashed] (b) -- (c);
    \draw [dashed] (a) -- (c);
    \draw [dashed] (a) -- (d);
    \draw [dashed] (b) -- (d);
    \draw [dashed] (c) -- (d);
    
    \draw [densely dashed] [color=blue] [line width=1.5pt] (a) .. controls (3.35, 4.5) .. (b);
    \draw [color=red] [line width=1.5pt] [line width=1.5pt] (a) .. controls (8.85, 3.5) .. (b);

    \draw [color=red] [line width=1.5pt] (a) .. controls (7, 1.25) .. (c);
    
    \draw [densely dashed] [color=blue] [line width=1.5pt] (a) .. controls (7.5, 4.25) .. (d);

    \draw [color=red] [line width=1.5pt] (b) .. controls (2, 6) .. (c);
    
    \draw [densely dashed] [color=blue] [line width=1.5pt] [line width=1.5pt] (b) .. controls (7.6, 5) .. (d);
   
    \node at (2.25,6)[myblue,label=below:$ $,fill=white, circle](C) {$C$};
    \node at (4.15,6)[myblue,label=below:$ $,fill=white, circle](D) {$D$};

  \end{tikzpicture}

%% file: Figures/lct=1-or-fenced-4.tex
  \begin{tikzpicture}
[scale=0.7, label distance=3pt, every node/.style={fill,circle,inner sep=0pt,minimum size=6pt}]
 
    \node at (6,3.4641)[myblue,label=below:$a$,fill=black, circle](a) {};
    \node at (6,10.3923)[myblue,label=above left:$b$,fill=black, circle](b) {};
    \node at (0,0)[myblue,label=above left :$c$,fill=black, circle](c) {};
    \node at (12,0)[myblue,label=above right:$d$,fill=black, circle](d) {};
    
    \draw [dashed] (a) -- (b);
    \draw [dashed] (b) -- (c);
    \draw [dashed] (a) -- (c);
    \draw [dashed] (a) -- (d);
    \draw [dashed] (b) -- (d);
    \draw [dashed] (c) -- (d);
    
    \draw [densely dashed] [color=blue] [line width=1.5pt] (a) .. controls (3.35, 4.5) .. (b);
    \draw [color=red] [line width=1.5pt] [line width=1.5pt] (a) .. controls (8.95, 3.5) .. (b);

    \draw [color=red] [line width=1.5pt] (a) .. controls (7, 1.25) .. (c);
    
    \draw [densely dashed] [color=blue] [line width=1.5pt] (a) .. controls (5, 1.25) .. (d);
    
    \draw [color=mygreen] [line width=1pt] (a) .. controls (4, 5.45) .. (c);
   
    \draw [color=mygreen] [line width=1pt] (a) .. controls (8, 5.45) .. (d);
    
    \draw [color=red] [line width=1.5pt] (b) .. controls (2, 6) .. (c);
    
    \draw [densely dashed] [color=blue] [line width=1.5pt] [line width=1.5pt] (b) .. controls (7.6, 5) .. (d);
    
    \draw [color=mygreen] [line width=1pt] (c) .. controls (6, 1.8) .. (d);
   
    \node at (2.25,6)[myblue,label=below:$ $,fill=white, circle](C) {$C$};
    \node at (4.15,6)[myblue,label=below:$ $,fill=white, circle](D) {$D$};

    \node at (3.15,4)[myblue,label=below:$ $,fill=white, circle](F) {$F$};
    
  \end{tikzpicture}

%% file: Figures/lct=1-or-fenced-5.tex
  \begin{tikzpicture}
[scale=0.7, label distance=3pt, every node/.style={fill,circle,inner sep=0pt,minimum size=6pt}]
 
    \node at (6,3.4641)[myblue,label=above right:$a$,fill=black, circle](a) {};
    \node at (6,10.3923)[myblue,label=above left:$b$,fill=black, circle](b) {};
    \node at (0,0)[myblue,label=above left :$c$,fill=black, circle](c) {};
    \node at (12,0)[myblue,label=above right:$d$,fill=black, circle](d) {};
    
    \draw [dashed] (a) -- (b);
    \draw [dashed] (b) -- (c);
    \draw [dashed] (a) -- (c);
    \draw [dashed] (a) -- (d);
    \draw [dashed] (b) -- (d);
    \draw [dashed] (c) -- (d);
    
    \draw [densely dashed] [color=blue] [line width=1.5pt] (a) .. controls (3.35, 4.5) .. (b);
    \draw [color=red] [line width=1.5pt] [line width=1.5pt] (a) .. controls (8.65, 4.5) .. (b);

    \draw [color=red] [line width=1.5pt] (a) .. controls (6, 1.25) .. (c);
    
    \draw [densely dashed] [color=blue] [line width=1.5pt] (a) .. controls (5, 1.25) .. (d);
    
    \draw [color=mygreen] [line width=1pt] (a) .. controls (4.3, 1.73) .. (d);
    
    \draw [color=mygreen] [line width=1pt] (a) .. controls (4, 5.45) .. (c);
   
    
    \draw [color=red] [line width=1.5pt] (b) .. controls (2, 6) .. (c);
    
    \draw [densely dashdotted] [color=purple][line width=1.5pt] (b) .. controls (4.7, 5) .. (c);
    
    \draw [densely dashed] [color=blue] [line width=1.5pt] [line width=1.5pt] (b) .. controls (6.4, 6) .. (d);
    \draw [densely dashdotted] [color=purple] [line width=1.5pt] [line width=1.5pt] (b) .. controls (7.73, 4) .. (d);
    
    \draw [color=mygreen] [line width=1pt] (c) .. controls (4.5, -1.5) and (3,-2.5) .. (7,-1) .. controls(10.3,0) and (10.5,-0.5) ..(d);
    
    \draw [densely dashdotted] [color=purple] [line width=1pt] (d) .. controls (7.5, -1.5) and (9,-2.5) .. (5,-1) .. controls(1.7,0) and (1.5,-0.5) ..(c);
   
    \node at (2.25,6)[myblue,label=below:$ $,fill=white, circle](C) {$C$};
    \node at (4.15,6)[myblue,label=below:$ $,fill=white, circle](D) {$D$};

    \node at (3.15,4)[myblue,label=below:$ $,fill=white, circle](F) {$F$};
    
    \node at (7.8,-1.75)[myblue,label=below:$ $,fill=white, circle](J) {$J$};
    
  \end{tikzpicture}

%% file: Figures/lct-remark_one-outside-4-inters-1.tex
  \begin{tikzpicture}
[scale=0.7, label distance=3pt, every node/.style={fill,circle,inner sep=0pt,minimum size=6pt}]
 
    \node at (6,3.4641)[myblue,label=above right:$a$,fill=black, circle](a) {};
    \node at (6,10.3923)[myblue,label=above left:$b$,fill=black, circle](b) {};
    \node at (0,0)[myblue,label=above left :$c$,fill=black, circle](c) {};
    \node at (12,0)[myblue,label= right:$d$,fill=black, circle](d) {};
    
    \draw [dashed] (a) -- (b);
    \draw [dashed] (b) -- (c);
    \draw [dashed] (a) -- (c);
    \draw [dashed] (a) -- (d);
    \draw [dashed] (b) -- (d);
    \draw [dashed] (c) -- (d);
    
    \draw [densely dashed] [color=blue] [line width=1.5pt] (a) .. controls (5, 1.25) .. (d);
    \draw [densely dashed] [color=blue] [line width=1.5pt] (a) .. controls (3.35, 4.5) .. (b);
    \draw  [densely dashed] [color=blue][line width=1.5pt]  (d) .. controls (9.9, 4) and (9.8,7) .. (9.5,9) .. controls(9.2,10) and (8,10) ..(b);

    \draw [color=red] [line width=1.5pt] (d) .. controls (14.5, 3.5) and (12,4) .. (12,4) .. controls(8,5) and (10,8.5) ..(b);
    \draw [color=red] [line width=1.5pt] [line width=1.5pt] (a) .. controls (8.65, 4.5) .. (b);
    \draw [color=red] [line width=1.5pt] (a) .. controls (6, 1.25) .. (c);
    \draw [color=red] [line width=1pt] (c) .. controls (6, -1) .. (d);
    
    \node at (7.75,6)[myblue,label=below:$ $,fill=white, circle](C) {$C$};
   
    \node at (4.15,6)[myblue,label=below:$ $,fill=white, circle](D) {$D$};
    
  \end{tikzpicture}

%% file: Figures/lct-remark_one-outside-4-inters-2.tex
  \begin{tikzpicture}
[scale=0.7, label distance=3pt, every node/.style={fill,circle,inner sep=0pt,minimum size=6pt}]
 
    \node at (6,3.4641)[myblue,label=above right:$a$,fill=black, circle](a) {};
    \node at (6,10.3923)[myblue,label=above left:$b$,fill=black, circle](b) {};
    \node at (0,0)[myblue,label=above left :$c$,fill=black, circle](c) {};
    \node at (12,0)[myblue,label= right:$d$,fill=black, circle](d) {};
    
    \draw [dashed] (a) -- (b);
    \draw [dashed] (b) -- (c);
    \draw [dashed] (a) -- (c);
    \draw [dashed] (a) -- (d);
    \draw [dashed] (b) -- (d);
    \draw [dashed] (c) -- (d);
    
    \draw [densely dashed] [color=blue] [line width=1.5pt] (a) .. controls (5, 1.25) .. (d);
    \draw [densely dashed] [color=blue] [line width=1.5pt] (a) .. controls (3.35, 4.5) .. (b);
    \draw [densely dashed] [color=blue] [line width=1.5pt] [line width=1.5pt] (b) .. controls (7.73, 4) .. (d);

    \draw [color=red] [line width=1.5pt] (d) .. controls (14.5, 3.5) and (12,4) .. (12,4) .. controls(8,5) and (10,8.5) ..(b);
    \draw [color=red] [line width=1.5pt] [line width=1.5pt] (a) .. controls (8.65, 4.5) .. (b);
    \draw [color=red] [line width=1.5pt] (a) .. controls (6, 1.25) .. (c);
    \draw [color=red] [line width=1pt] (c) .. controls (6, -1) .. (d);
    
    \node at (7.75,6)[myblue,label=below:$ $,fill=white, circle](C) {$C$};
   
    \node at (4.15,6)[myblue,label=below:$ $,fill=white, circle](D) {$D$};
    
  \end{tikzpicture}

%% file: Figures/lct-remark_one-outside-4-inters-3.tex
  \begin{tikzpicture}
[scale=0.7, label distance=3pt, every node/.style={fill,circle,inner sep=0pt,minimum size=6pt}]
 
    \node at (6,3.4641)[myblue,label=below left:$ $,fill=black, circle](a) {};
    \node at (6,10.3923)[myblue,label=above left:$b$,fill=black, circle](b) {};
    \node at (0,0)[myblue,label=above left :$c$,fill=black, circle](c) {};
    \node at (12,0)[myblue,label= right:$d$,fill=black, circle](d) {};
    
    \draw [dashed] (a) -- (b);
    \draw [dashed] (b) -- (c);
    \draw [dashed] (a) -- (c);
    \draw [dashed] (a) -- (d);
    \draw [dashed] (b) -- (d);
    \draw [dashed] (c) -- (d);
    
    \draw [densely dashed] [color=blue][line width=1.5pt](a) .. controls (8, 5.45) .. (d);
    \draw [densely dashed] [color=blue] [line width=1.5pt] (a) .. controls (3.35, 4.5) .. (b);
    \draw  [densely dashed] [color=blue][line width=1.5pt]  (d) .. controls (9.9, 4) and (9.8,7) .. (9.5,9) .. controls(9.2,10) and (8,10) ..(b);

    \draw [color=red] [line width=1.5pt] (d) .. controls (14.5, 3.5) and (12,4) .. (12,4) .. controls(8,5) and (10,8.5) ..(b);
    \draw [color=red] [line width=1.5pt] [line width=1.5pt] (a) .. controls (8.65, 4.5) .. (b);
    \draw [color=red] [line width=1.5pt] (a) .. controls (6, 1.25) .. (c);
    \draw [color=red] [line width=1pt] (c) .. controls (6, -1) .. (d);
    
    \draw [color=mygreen] [line width=1pt] (a) .. controls (4, 5.45) .. (c);
    \draw [color=mygreen] [line width=1pt] (a) .. controls (8, 1.45) .. (d);
    \draw [color=mygreen] [line width=1pt] (c) .. controls (4.5, 1.5) and (3,2.5) .. (7,1) .. controls(10.3,0) and (10.5,0.5) ..(d);
    
    \node at (7.75,6)[myblue,label=below:$ $,fill=white, circle](C) {$C$};
   
    \node at (4.15,6)[myblue,label=below:$ $,fill=white, circle](D) {$D$};
    \node at (5.65,3.5)[myblue,label=below:$a$,fill=none, circle](x) {};
    \node at (3.15,4)[myblue,label=below:$ $,fill=white, circle](F) {$F$};
    
  \end{tikzpicture}
 

%% file: Figures/lct-remark-0-3crosses-1.tex
  \begin{tikzpicture}
[scale=0.7, label distance=3pt, every node/.style={fill,circle,inner sep=0pt,minimum size=6pt}]
 
    \node at (6,3.4641)[myblue,label=below :$a$,fill=black, circle](a) {};
    \node at (6,10.3923)[myblue,label=above left:$b$,fill=black, circle](b) {};
    \node at (0,0)[myblue,label=above left :$c$,fill=black, circle](c) {};
    \node at (12,0)[myblue,label= right:$d$,fill=black, circle](d) {};
    
    \draw [dashed] (a) -- (b);
    \draw [dashed] (b) -- (c);
    \draw [dashed] (a) -- (c);
    \draw [dashed] (a) -- (d);
    \draw [dashed] (b) -- (d);
    \draw [dashed] (c) -- (d);
    
    \draw [color=red] [line width=1.5pt] [line width=1.5pt] (a) .. controls (8.95, 3.5) .. (b);
    \draw  [color=red] [line width=1.5pt] (a) .. controls (4.5, 4.25) .. (c);
    \draw  [color=red] [line width=1.5pt] (b) .. controls (4.4, 5) .. (c);
     
    \draw [densely dashed] [color=blue] [line width=1.5pt] (b) .. controls (7.6, 5) .. (d);
    \draw [densely dashed] [color=blue] [line width=1.5pt] (a) .. controls (3.35, 3.8) .. (b);
    \draw [densely dashed] [color=blue] [line width=1.5pt] (a) .. controls (5, 1.25) .. (d);
    
    \node at (3.3,4)[myblue,label=below:$ $,fill=white, circle](C) {$C$};
    \node at (6,1.35)[myblue,label=below:$ $,fill=white, circle](D) {$D$};
    
  \end{tikzpicture}

%% file: Figures/lct-remark-0-3crosses-2.tex
  \begin{tikzpicture}
[scale=0.7, label distance=3pt, every node/.style={fill,circle,inner sep=0pt,minimum size=6pt}]
 
    \node at (6,3.4641)[myblue,label=below :$a$,fill=black, circle](a) {};
    \node at (6,10.3923)[myblue,label=above left:$b$,fill=black, circle](b) {};
    \node at (0,0)[myblue,label=above left :$c$,fill=black, circle](c) {};
    \node at (12,0)[myblue,label= right:$d$,fill=black, circle](d) {};
    
    \draw [dashed] (a) -- (b);
    \draw [dashed] (b) -- (c);
    \draw [dashed] (a) -- (c);
    \draw [dashed] (a) -- (d);
    \draw [dashed] (b) -- (d);
    \draw [dashed] (c) -- (d);
    
    \draw [color=red] [line width=1.5pt] [line width=1.5pt] (a) .. controls (8.95, 3.5) .. (b);
    \draw  [color=red] [line width=1.5pt] (a) .. controls (4.5, 4.25) .. (c);
    \draw  [color=red] [line width=1.5pt] (b) .. controls (4.4, 5) .. (c);
     
    \draw [densely dashed] [color=blue] [line width=1.5pt] (b) .. controls (7.6, 5) .. (d);
    \draw [densely dashed] [color=blue] [line width=1.5pt] (a) .. controls (5, 1.25) .. (d);
    \draw [densely dashed] [color=blue] [line width=1.5pt] (a) .. controls (7.5, 4.5) .. (b);
    
    \draw [color=mygreen] [line width=1pt] (a) .. controls (3.35, 3.8) .. (b);
    \draw [color=mygreen] [line width=1pt] (a) .. controls (4, 3.25) .. (c);
    \draw [color=mygreen] [line width=1pt] (b) .. controls (2, 6) .. (c);
    
    \node at (3.3,4)[myblue,label=below:$ $,fill=white, circle](C) {$C$};
    \node at (6,1.35)[myblue,label=below:$ $,fill=white, circle](D) {$D$};
    \node at (2.25,6)[myblue,label=below:$ $,fill=white, circle](F) {$F$};
    
  \end{tikzpicture}

%% file: Figures/lct-remark-0-3crosses-3.tex
  \begin{tikzpicture}
[scale=0.7, label distance=3pt, every node/.style={fill,circle,inner sep=0pt,minimum size=6pt}]
 
    \node at (6,3.4641)[myblue,label=below :$a$,fill=black, circle](a) {};
    \node at (6,10.3923)[myblue,label=above left:$b$,fill=black, circle](b) {};
    \node at (0,0)[myblue,label=above left :$c$,fill=black, circle](c) {};
    \node at (12,0)[myblue,label= right:$d$,fill=black, circle](d) {};
    
    \draw [dashed] (a) -- (b);
    \draw [dashed] (b) -- (c);
    \draw [dashed] (a) -- (c);
    \draw [dashed] (a) -- (d);
    \draw [dashed] (b) -- (d);
    \draw [dashed] (c) -- (d);
    
    \draw [color=red] [line width=1.5pt] [line width=1.5pt] (a) .. controls (8.95, 3.5) .. (b);
    \draw  [color=red] [line width=1.5pt] (a) .. controls (4.5, 4.25) .. (c);
    \draw  [color=red] [line width=1.5pt] (b) .. controls (4.4, 5) .. (c);

    \draw [color=mygreen] [line width=1pt] (a) .. controls (2.5, 3.6) .. (c);
    \draw [color=mygreen] [line width=1pt] (a) .. controls (5, 1.25) .. (d);
    \draw [color=mygreen] [line width=1pt] (c) .. controls (6, -1) .. (d);
    
    \node at (4.3,5.4)[myblue,label=below:$ $,fill=white, circle](C) {$C$};
    \node at (6,1.35)[myblue,label=below:$ $,fill=white, circle](F) {$F$};
    
    \node at (2.8,3.3)[myblue,label=above:$x$,fill=black, circle](x) {};

  \end{tikzpicture}

%% file: Figures/lct-remark-1-3crosses.tex
  \begin{tikzpicture}
[scale=0.7, label distance=3pt, every node/.style={fill,circle,inner sep=0pt,minimum size=6pt}]
 
    \node at (6,3.4641)[myblue,label=below:$a$,fill=black, circle](a) {};
    \node at (6,10.3923)[myblue,label=above left:$b$,fill=black, circle](b) {};
    \node at (0,0)[myblue,label=above left :$c$,fill=black, circle](c) {};
    \node at (12,0)[myblue,label=above right:$d$,fill=black, circle](d) {};
    
    \draw [dashed] (a) -- (b);
    \draw [dashed] (b) -- (c);
    \draw [dashed] (a) -- (c);
    \draw [dashed] (a) -- (d);
    \draw [dashed] (b) -- (d);
    \draw [dashed] (c) -- (d);
    
    \draw [color=red] [line width=1.5pt] [line width=1.5pt] (a) .. controls (8.65, 4.5) .. (b);
    \draw [color=red] [line width=1.5pt] (a) .. controls (7, 1.25) .. (c);
    \draw [color=red] [line width=1.5pt] (b) .. controls (2, 6) .. (c);
    
    \draw [densely dashed] [color=blue] [line width=1.5pt] (a) .. controls (3.35, 4.5) .. (b);
    \draw [densely dashed] [color=blue] [line width=1.5pt] (a) .. controls (5, 1.25) .. (d);
    \draw [densely dashed] [color=blue] [line width=1.5pt] (b) .. controls (10, 6) .. (d);
    
    \draw [color=mygreen] [line width=1pt] (b) .. controls (4.7, 5) .. (c);
    \draw [color=mygreen] [line width=1pt] (c) .. controls (6, 1.8) .. (d);
    \draw [color=mygreen] [line width=1pt] (b) .. controls (7.3, 5) .. (d);
    
    \draw  [densely dashdotted] [color=purple][line width=1.5pt] (a) .. controls (4, 5.45) .. (c);
    \draw  [densely dashdotted] [color=purple][line width=1.5pt] (a) .. controls (8, 5.45) .. (d);
    \draw  [densely dashdotted] [color=purple][line width=1.5pt] (c) .. controls (6, -1) .. (d);
    
    \node at (2.25,6)[myblue,label=below:$ $,fill=white, circle](X) {$C$};
    \node at (9.75,6)[myblue,label=below:$ $,fill=white, circle](D) {$D$};
    \node at (6,-0.75)[myblue,label=below:$ $,fill=white, circle](H) {$H$};
    \node at (4.75,6)[myblue,label=below:$ $,fill=white, circle](F) {$F$};
    
  \end{tikzpicture}

%% file: Figures/lct-remark-2-3crosses-4.tex
  \begin{tikzpicture}
[scale=0.7, label distance=3pt, every node/.style={fill,circle,inner sep=0pt,minimum size=6pt}]
 
    \node at (6,3.4641)[myblue,label=above right:$a$,fill=black, circle](a) {};
    \node at (6,10.3923)[myblue,label=above left:$b$,fill=black, circle](b) {};
    \node at (0,0)[myblue,label= left :$c$,fill=black, circle](c) {};
    \node at (12,0)[myblue,label=above right:$d$,fill=black, circle](d) {};
    
    \draw [dashed] (a) -- (b);
    \draw [dashed] (b) -- (c);
    \draw [dashed] (a) -- (c);
    \draw [dashed] (a) -- (d);
    \draw [dashed] (b) -- (d);
    \draw [dashed] (c) -- (d);
    
    \draw [color=red] [line width=1.5pt] [line width=1.5pt] (a) .. controls (8.65, 4) .. (b);
    \draw [color=red] [line width=1.5pt] (a) .. controls (7, 1.25) .. (c);
    \draw [color=red] [line width=1.5pt] (b) .. controls (4.5, 3.75) .. (c);

     \draw [densely dashed] [color=blue] [line width=1.5pt] (b) .. controls (5, 3.5) .. (c);
    \draw [densely dashed] [color=blue][line width=1.5pt] (b) .. controls (10, 6) .. (d);
     \draw [densely dashed] [color=blue][line width=1.5pt] (c) .. controls (6, 2) .. (d);
     
     \draw [color=mygreen] [line width=1pt]  (c) .. controls (6, -1) .. (d);
    \draw [color=mygreen][line width=1pt] [line width=1pt] (b) .. controls (7.73, 4) .. (d);
     \draw [color=mygreen] [line width=1pt] (b) .. controls (4.3, 5) .. (c);

    \node at (10.35,4.75)[myblue,label=below:$ $,fill=white, circle](F) {$F$};
    \node at (7,3.7)[myblue,label=below:$ $,fill=white, circle](C) {$C$};
    \node at (6,-0.8)[myblue,label=below:$ $,fill=white, circle](C) {$H$};
    
  \end{tikzpicture} 

%% file: Figures/lct-3crosses-J-case2.tex
\begin{tikzpicture}
[scale=0.7, label distance=3pt, every node/.style={fill,circle,inner sep=0pt,minimum size=6pt}]
 
    \node at (6,3.4641)[myblue,label=above right:$a$,fill=black, circle](a) {};
    \node at (6,10.3923)[myblue,label=above left:$b$,fill=black, circle](b) {};
    \node at (0,0)[myblue,label= left :$c$,fill=black, circle](c) {};
    \node at (12,0)[myblue,label=above right:$d$,fill=black, circle](d) {};
    
    \draw [dashed] (a) -- (b);
    \draw [dashed] (b) -- (c);
    \draw [dashed] (a) -- (c);
    \draw [dashed] (a) -- (d);
    \draw [dashed] (b) -- (d);
    \draw [dashed] (c) -- (d);
    
    \draw [color=red] [line width=1.5pt] [line width=1.5pt] (a) .. controls (8.65, 4) .. (b);
    \draw [color=red] [line width=1.5pt] (a) .. controls (6.2, 1.25) .. (c);
    \draw [color=red] [line width=1.5pt] (b) .. controls (4.5, 3.75) .. (c);

     \draw [densely dashed] [color=blue] [line width=1.5pt] (b) .. controls (5, 3.5) .. (c);
    \draw [densely dashed] [color=blue][line width=1.5pt] (b) .. controls (10, 6) .. (d);
     \draw [densely dashed] [color=blue][line width=1.5pt] (c) .. controls (6, 2) .. (d);
     
     \draw [color=mygreen] [line width=1pt]  (c) .. controls (6, -1) .. (d);
    \draw [color=mygreen][line width=1pt] [line width=1pt] (b) .. controls (7.73, 4) .. (d);
     \draw [color=mygreen] [line width=1pt] (b) .. controls (4.3, 5) .. (c);


    \draw [densely dashdotted] [color=purple][line width=1.5pt] (b) .. controls (6.3, 4) .. (d);
    \draw  [densely dashdotted] [color=purple][line width=1.5pt] (b) .. controls (2, 6) .. (c);
    
    \draw  [densely dashdotted] [color=purple][line width=1.5pt] (c) .. controls (4.5, 1.5) and (3,2.5) .. (7,1) .. controls(10.3,0) and (10.5,0.5) ..(d);
    
    \node at (10.35,4.75)[myblue,label=below:$ $,fill=white, circle](F) {$F$};
    \node at (6,2.6)[myblue,label=below:$ $,fill=white, circle](C) {$C$};
    \node at (6,-0.8)[myblue,label=below:$ $,fill=white, circle](C) {$H$};
    \node at (1.65,4.75)[myblue,label=below:$ $,fill=white, circle](C) {$J$};
    
  \end{tikzpicture} 

%% file: Figures/lct-3crosses-J-case1.tex
\begin{tikzpicture}
[scale=0.7, label distance=3pt, every node/.style={fill,circle,inner sep=0pt,minimum size=6pt}]
 
    \node at (6,3.4641)[myblue,label=above left:$a$,fill=black, circle](a) {};
    \node at (6,10.3923)[myblue,label=above left:$b$,fill=black, circle](b) {};
    \node at (0,0)[myblue,label= left :$c$,fill=black, circle](c) {};
    \node at (12,0)[myblue,label=above right:$d$,fill=black, circle](d) {};
    
    \draw [dashed] (a) -- (b);
    \draw [dashed] (b) -- (c);
    \draw [dashed] (a) -- (c);
    \draw [dashed] (a) -- (d);
    \draw [dashed] (b) -- (d);
    \draw [dashed] (c) -- (d);
    
    \draw [color=red] [line width=1.5pt] [line width=1.5pt] (a) .. controls (8.65, 4) .. (b);
    \draw [color=red] [line width=1.5pt] (a) .. controls (6.2, 1.25) .. (c);
    \draw [color=red] [line width=1.5pt] (b) .. controls (4.5, 3.75) .. (c);

     \draw [densely dashed] [color=blue] [line width=1.5pt] (b) .. controls (5, 3.5) .. (c);
     \draw [densely dashed] [color=blue][line width=1.5pt] (c) .. controls (6, 2) .. (d);
    \draw [densely dashed] [color=blue][line width=1.5pt] (d) .. controls (9.9, 4) and (9.8,7) .. (9.5,9) .. controls(9.2,10) and (8,10) ..(b);

     \draw [color=mygreen] [line width=1pt]  (c) .. controls (6, -1) .. (d);
    \draw [color=mygreen][line width=1pt] [line width=1pt] (b) .. controls (7.73, 4) .. (d);
     \draw [color=mygreen] [line width=1pt] (b) .. controls (4.3, 5) .. (c);

    \draw  [densely dashdotted] [color=purple] [line width=1.5pt] (a) .. controls (5, 1.25) .. (d);
    
    \draw [densely dashdotted] [color=purple] [line width=1.5pt] (a) .. controls (8, 4.5) .. (b);
    \draw [densely dashdotted] [color=purple] [line width=1.5pt] (d) .. controls (14.5, 3.5) and (12,4) .. (12,4) .. controls(8,5) and (10,8.5) ..(b);
    
    \node at (9.75,6.75)[myblue,label=below:$ $,fill=white, circle](F) {$F$};
    \node at (6,2.6)[myblue,label=below:$ $,fill=white, circle](C) {$C$};
    \node at (6,-0.8)[myblue,label=below:$ $,fill=white, circle](C) {$H$};
    \node at (12.2,4)[myblue,label=below:$ $,fill=white, circle](C) {$J$};
    
  \end{tikzpicture} 

%% file: sectionAuxiliaryLemmas.tex
The main result of this section is Corollary
\ref{corollary:lct-tw3-lct>1-implies-not-jump},
which is used many times in Section \ref{section:mainlemma}.
For all lemmas, propositions and corollaries in this subsection,
we fix a graph~$G$ with~$\tw(G)=3$,
a full tree decomposition~$(T,\calV)$ of~$G$, and
a node~$t \in V(T)$ with $V_t=\{a,b,c,d\}$.
For every~$ij \in \{ab,bc,ac\}$,
we set~$\cC_{ij}$ as the collection of all longest cycles in~$G$ that 2-jump
$abc$ at~$\{i,j\}$
and put~$\cC_{abc}$ as the collection of all longest cycles in~$G$ that 3-jump~$abc$.
Finally, we set~${\cC = \cC_{ab} \cup \cC_{bc} \cup \cC_{ac} \cup \cC_{abc}}$.

We begin by extending the definition of branch.
Given a component $A$ of $G-V_t$, we say that
$\Branch_t(A)=\Branch_t(v)$, where $v$ is a vertex in $A$.
Observe that $\Branch_t(A)$ is well defined, as every pair of vertices in~$A$
are not separated by $V_t$.

Given a triple of vertices~${\Delta \subseteq V_t}$,
it is denoted by~$\A_t(\Delta)$ the set of components inside~$\Delta$,
that is,~$${\A_t(\Delta)=\{A: A \mbox{ is a component of } {G-V_t} \mbox{ and }
\Branch_t(A)\subseteq \cB_t(\Delta)\}}.$$
(Recall the definition of $\cB_t(\Delta)$ in Section \ref{section:basic-concepts-treewidth}.)

\begin{lemma} \label{lemma:lct-tw3-pairwise-intersect}
If~$\cC_{ij}\neq \emptyset$ for every~$ij \in \{ab,bc,ac\}$, then
all longest cycles in~$\cC$ pairwise intersect each other in a
component of~$\A_t(abc)$.
\end{lemma}
\begin{proof}
Let~$(\bar{C},\bar{D}) \in \cC_{ab} \times \cC_{ac}$.
Let~$\bar{C}'$ and~$\bar{C}''$ be the two~$ab$-parts of~$\bar{C}$.
Let~$\bar{D}'$ and~$\bar{D}''$ be the two~$ac$-parts of~$\bar{D}$.
As~$\bar{C}$ and~$\bar{D}$ jump~$abc$, we may assume that~$\bar{C}'$
and~$\bar{D}'$ are inside~$abc$ and that~$\bar{C}''$ and~$\bar{D}''$ are outside~$abc$.
Suppose for a moment that~$\bar{C}'$ and~$\bar{D}'$ are internally disjoint.
Let~$F \in \cC_{bc}$, which exists by the hypothesis of the lemma.
Let~$F'$ and~$F''$ be the two~$bc$-parts of~$F$.
As~$F$ jumps~$abc$, we may assume
that~$F'$ is inside~$abc$ and that~$F''$ is outside~$abc$.
Let~$R$ be a subpath of~$F'$ that is internally disjoint
from~$\bar{C}'$ and~$\bar{D}'$ connecting~$\bar{C}'$ and~$\bar{D}'$.
Suppose that~$V(R) \cap V(\bar{C})=\{x\}$ and that~$V(R) \cap V(\bar{D})=\{y\}$.
Then~$\bar{C}'' \cdot \bar{C}'_{bx} \cdot R \cdot \bar{D}'_{ya}$
and~$\bar{D}'' \cdot \bar{D}'_{cy} \cdot R \cdot \bar{C}'_{xa}$
are both cycles,
one of them longer than~$L$,
a contradiction (Figure~\ref{fig:lct-tw3-pairw-inters-1}(a)).

Thus~$\bar{C}'$ and~$\bar{D}'$ internally intersect inside~$abc$.
As~$\bar{C}'$ and~$\bar{D}'$ {2-intersects}~$abc$, they are fenced by~$abc$.
Hence, as they internally intersect each other, there exists a component
$A \in \A_t(abc)$ such that 
all internal vertices of~$\bar{C}'$ and~$\bar{D}'$ are in~$A$.
We show that~$A$ is the desired component.
That is, all longest cycles in~$\cC$ pairwise intersect each other in~$A$.

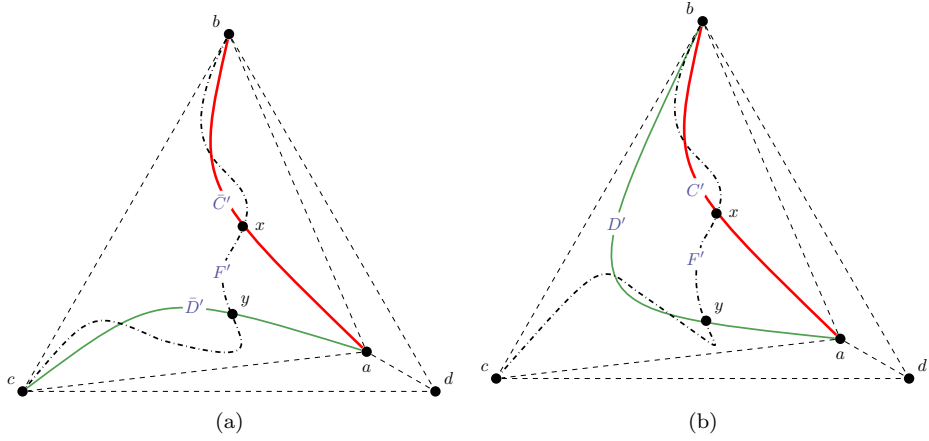
\begin{figure}[H]
\centering
\subfigure[ ]{
\resizebox{.5\textwidth}{!}{\input{Figures/lct-pairw-inters-1}}
}
\subfigure[ ]{
\resizebox{.5\textwidth}{!}{\input{Figures/lct-pairw-inters-2}}
}
\caption{Cases in the proof of Lemma~\ref{lemma:lct-tw3-pairwise-intersect}. 
(a) Pair~$(\bar{C},\bar{D}) \in \cC_{ab} \times \cC_{ac}$ with~$\bar{C}'$ and~$\bar{D}'$ inside~$abc$ and the path~$F'$.
(b) Pair~$(C,D) \in \cC_{ab} \times \cC_{ab}$ with~$C'$ and~$D'$ inside~$abc$ and the path~$F'$.}
\label{fig:lct-tw3-pairw-inters-1}
\end{figure}

Let~$(C,D) \in \cC_{ab} \times \cC_{ac}$.
Let~$C'$ and~$C''$ be the two~$ab$-parts of~$C$.
Let~$D'$ and~$D''$ be the two~$ac$-parts of~$D$.
Analogously to the previous paragraph,~$C'$ and~$D'$ internally intersect
inside~$abc$.
Also,~$C'$ and~$\bar{D}'$ internally intersect inside~$abc$.
Hence the internal vertices of~$C'$ and~$D'$ are in~$A$.
We conclude that~$C'$ and~$D'$ intersect in a vertex of~$A$.
A similar proof shows that every pair of cycles in~$(\cC_{ab} \times \cC_{bc}) \cup (\cC_{bc} \times \cC_{ac})$
intersect each other in a vertex of~$A$.

Let~$(C,D) \in \cC_{ab} \times \cC_{ab}$.
Let~$C'$ and~$C''$ be the two~$ab$-parts of~$C$.
Let~$D'$ and~$D''$ be the two~$ab$-parts of~$D$,
with $C'$ and $D'$ inside $abc$.
By the previous paragraph,~$C'$ and~$\bar{D}'$ internally intersect
in~$A$, and~$D'$ and~$\bar{D}'$ internally intersect
in~$A$. Hence, all internal vertices of $C'$ and~$D'$ are in~$A$.
So, it suffices to prove that 
$C'$ and~$D'$ internally intersect
inside~$abc$.

Suppose for a moment that~$C'$ and~$D'$ are internally disjoint.
Let~$F \in \cC_{bc}$, which exists by the hypothesis of the lemma.
Let~$R$ be a subpath of~$F$ that is inside $abc$, internally disjoint
from~$C'$ and~$D'$ and connecting~$C'$ and~$D'$. 
Let~${\{x\}=V(R) \cap V(C')}$ and~${\{y\}=V(R) \cap V(D')}$.
Then~${C'' \cdot C'_{bx}  \cdot R \cdot D'_{ya}}$ 
and~${D'' \cdot D'_{by} \cdot R \cdot C'_{xa}}$
are both cycles,
one of them longer than~$L$,
a contradiction (Figure~\ref{fig:lct-tw3-pairw-inters-1}(b)).
A similar proof shows that every pair of
cycles in~${(\cC_{bc} \times \cC_{bc}) \cup (\cC_{ac} \times \cC_{ac})}$
intersect in~$A$.

Let~$(C,D) \in \cC_{abc} \times \cC_{ab}$.
Let~$D'$ and~$D''$ be the two~$ab$-parts of~$D$, with~$D'$ inside~$abc$.
By the previous paragraph,~$D'$ and~$\bar{C}'$ internally intersect
in~$A$. Hence, all internal vertices of~$D'$ are in~$A$.
So, it suffices to prove that 
$D'$ internally intersects any of~$\{C_{ab}, C_{bc}, C_{ac}\}$.
Suppose not.
As~$C$ jumps~$abc$, at least one of~$\{C_{ab}, C_{bc}, C_{ac}\}$
is inside~$abc$.
If~$C_{ab}$ is inside~$abc$ then
${C_{ac} \cdot C_{cb} \cdot D'}$ and~${D'' \cdot C_{ab}}$ are longest cycles.
As~${D'' \cdot C_{ab} \in \cC_{ab}}$,
by the previous paragraph, 
${D'' \cdot C_{ab}}$ intersects~$D$ in a vertex of~$A$.
As~$D''$ is outside~$abc$,~$C_{ab}$ and~$D'$ internally intersect
in~$A$, a contradiction.
Hence~$C_{ab}$ is outside~$abc$.
Let~${F \in \cC_{bc}}$.
By considering a subpath $R$ of $F$ as in
Figure~\ref{fig:lct-tw3-pairw-inters-2}(a), we obtain
two cycles~${C_{ab} \cdot D'_{bx} \cdot R \cdot C_{yc} \cdot C_{ca}}$
and~${D'' \cdot C_{by} \cdot R \cdot D'_{xa}}$
a contradiction.
Hence, every pair of cycles in~$\cC_{abc} \times \cC_{ab}$
intersect in~$A$.
A similar proof shows that every 
pair of cycles in~$(\cC_{abc} \times \cC_{bc}) \cup (\cC_{abc} \times \cC_{ac})$
intersect in~$A$.

\begin{figure}[hbt]
\centering
\subfigure[ ]{
\resizebox{.5\textwidth}{!}{\input{Figures/lct-pairw-inters-3}}
}
\subfigure[ ]{
\resizebox{.5\textwidth}{!}{\input{Figures/lct-pairw-inters-4}}
}
\caption{
Cases in the proof of Lemma~\ref{lemma:lct-tw3-pairwise-intersect}.
(a)~Pair~$(C,D) \in \cC_{abc} \times \cC_{ab}$, with~$C_{bc}$ and~$D'$ inside~$abc$, and the path~$F'$.
(b)~Pair~$(C,D) \in \cC_{abc} \times \cC_{abc}$
with~$C_{ab}, C_{bc}$,~$D_{ab}$ and~$D_{bc}$ inside~$abc$, and the path~$F'$.}
\label{fig:lct-tw3-pairw-inters-2}
\end{figure}
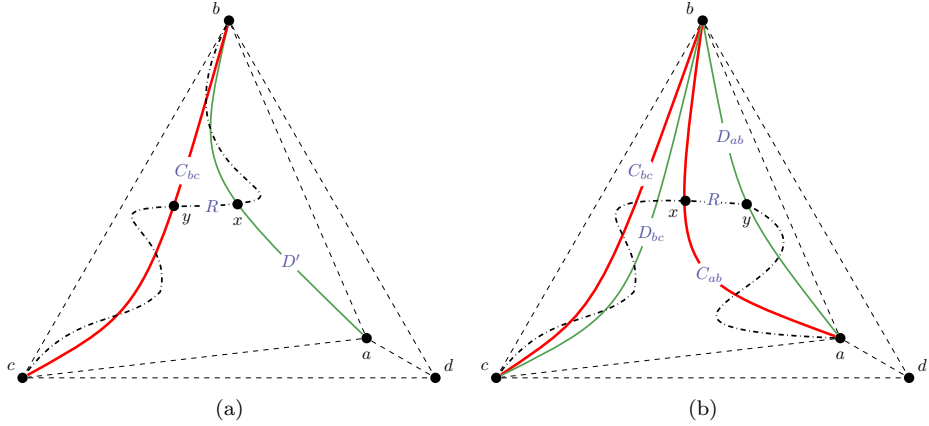

Finally, let~$(C,D) \in \cC_{abc} \times \cC_{abc}$.
Suppose that only one of~$\{C_{ab}, C_{ac}, C_{bc}\}$, say~$C_{ab}$, is inside~$abc$.
Let~$F \in \cC_{ab}$,~$F'$ be the~$ab$-part of~$F$ that is inside~$abc$,
and~$F''$ be the~$ab$-part of~$F$ that is outside~$abc$.
Then~$C_{ac} \cdot C_{cb} \cdot F'$ and~$F'' \cdot C_{ab}$ are longest cycles.
By the previous paragraph, as~$F'' \cdot C_{ab} \in \cC_{ab}$,~$C_{ab}$
and~$D$ intersect in a vertex of~$A$.
Hence, we may assume that two of
$\{C_{ab}, C_{ac}, C_{bc}\}$ are inside~$abc$, and, analogously,
two of~$\{D_{ab}, D_{ac}, D_{bc}\}$ are inside~$abc$.
Without loss of generality we have two cases.

First suppose that~$C_{ab}$, $C_{bc}$,
$D_{ab}$, and $D_{bc}$ are inside~$abc$.
Let~$F \in \cC_{ac}$.
By considering a subpath $R$ of $F$ as in
Figure~\ref{fig:lct-tw3-pairw-inters-2}(b), we obtain two cycles~$(C-C'_{bx}) \cdot R \cdot D'_{yb}$ and~$(D-D'_{by}) \cdot R \cdot C'_{xb}$
are cycles, 
a contradiction.

Now suppose that~$\{C_{ab}, C_{bc}\}$ and
$\{D_{ab}, D_{ac}\}$ are inside~$abc$.
Let~$F \in \cC_{ab}$.
By considering a subpath $R$ of $F$
that starts
at a vertex~$x$ of~$C$ and finishes
at a vertex~$y$ of~$D$,
we have two cases
as in Figure~\ref{fig:lct-tw3-pairw-inters-3}
In the former, we obtain two cycles
${(C-C_{ax})\cdot R \cdot D'_{ya}}$ and~${(D-D'_{ay}) \cdot R \cdot C_{xa}}$; in the later
we obtain two cycles
$(C-C_{cx})\cdot R \cdot D_{yc}$ and~$(D-D_{cy}) \cdot R \cdot C_{xc}$, a contradiction.

 \begin{figure}[hbt]
\centering
\subfigure[ ]{
\resizebox{.5\textwidth}{!}{\input{Figures/lct-pairw-inters-5}}
}
\subfigure[ ]{
\resizebox{.5\textwidth}{!}{\input{Figures/lct-pairw-inters-6}}
}
\caption{Pair~$(C,D) \in \cC_{abc} \times \cC_{abc}$,
with~$C_{ab}, C_{bc}, D_{ab}$ and~$D_{ac}$ inside~$abc$.
(a)~$x \in C_{ab}$ and~$y \in D_{ab}$.
(b)~$x \in C_{bc}$ and~$y \in D_{ac}$.}
\label{fig:lct-tw3-pairw-inters-3}
\end{figure}
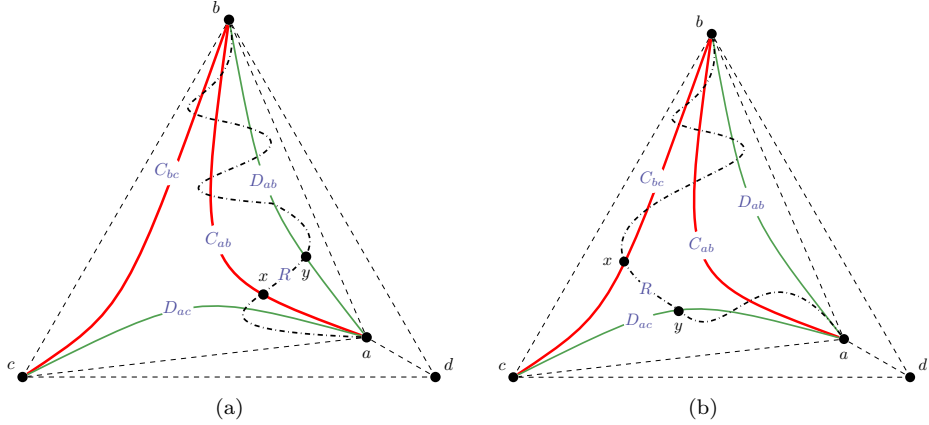

As~$\cC \times \cC = (\cC_{ab} \times \cC_{ab}) \cup (\cC_{ac} \times \cC_{ac}) \cup
(\cC_{bc} \times \cC_{bc}) \cup (\cC_{ab} \times \cC_{ac}) \cup (\cC_{ab} \times \cC_{bc}) \cup
{(\cC_{bc} \times \cC_{ac})} \cup (\cC_{abc} \times \cC_{ab}) \cup (\cC_{abc} \times \cC_{ac})
\cup (\cC_{abc} \times \cC_{bc}) \cup (\cC_{abc} \times \cC_{abc})$, the proof follows.
\end{proof}

Given a component~$A$ of~$G-V_t$, it is denoted by
~$t^{*}(A)$
the neighbor of $t$
such that $\Branch_t(A)=\Branch_t(t^{*}(A))$.



\begin{lemma} \label{lemma:lct-tw3-invariant}
If~$\cC_{ij}\neq \emptyset$ for every~$ij \in \{ab,bc,ac\}$,
then all cycles in~$\cC$ intersect at a common 
vertex~$t$-inside~$abc$.
\end{lemma}
\begin{proof}
By Lemma~\ref{lemma:lct-tw3-pairwise-intersect}, there exists a component
$A \in \A_t(abc)$ such that
all cycles in~$\cC$ pairwise intersect in~$A$.
Let~$t^{*}:=t^{*}(A)$.
As~$\cC_{ij}\neq \emptyset$ for every~${ij \in \{ab,bc,ac\}}$,
there exists an~$ij$-part~$\bar{C}_{ij}$ of a cycle in $\cC_{ij}$
such that
all internal vertices of~$\bar{C}_{ij}$ are in~$A$, for every~${ij \in \{ab,bc,ac\}}$.
Let~$V_{t^{*}}=\{a,b,c,f\}$.
We proceed by induction on the number of vertices~$t$-inside~$abc$
different from $a$,~$b$, and $c$.
If~$f$ is the only vertex~$t$-inside~$abc$ different from~$a$,~$b$, and~$c$, then,
as all cycles in~$\cC$ pairwise intersect in~$A$,~$f$ is in all cycles in~$\cC$ and the proof follows.
So, there exist at least two vertices~$t$-inside~$abc$
different from~$a$,~$b$, and~$c$.
Suppose by contradiction that there is no vertex in~$A$
that belongs to all cycles in~$\cC$.
Then, there exists (at least) one longest cycle in~$\cC$ that does not contain~$f$.
Let~$C$ be such a cycle. Without loss of generality, assume that~$C \in \cC_{ab}$. The proof when
~$C \in \cC_{abc}$ is similar.




Let~$C$ be such a cycle. Without loss of generality, assume that~$C \in \cC_{ab}$.
Let~$C'$ and~$C''$ be the two~$ab$-parts of~$C$,
with
$C'$~$t$-inside~$abc$.
Let~$A'$ be the component of $G-{V_{t^{*}}}$
where the internal vertices of $C'$ lie. 
It is straightforward to see that
$A' \in \A_{t^{*}}(\mathit{abf})$.

As every cycle in $\cC$ must intersect with
$C$ in a vertex of $A'$, we can conclude that
\begin{equation}\label{eq:every-cycle-in-cC-t'-jumps-abf}
\mbox{ every cycle in } \cC \mbox{ }t^{*}\mbox{-}\mbox{ jumps } \mathit{abf}.
\end{equation}

Note also that 
every cycle in $\cC_{ac}$
$t^{*}$-jumps~$\mathit{abf}$ at~$\{a,f\}$,
every cycle in $\cC_{bc}$
$t^{*}$-jumps~$\mathit{abf}$ at~$\{b,f\}$.
We conclude that~${\cC}_{ij}^{'} \neq \emptyset$ for every~$ij \in \{ab, \mathit{bf}, af\}$,
where~${\cC}^{'}_{ij}$ is the collection of all longest cycles that 2-jumps~$\mathit{abf}$ at~$\{i,j\}$.

Note that every vertex $t^{*}$-inside~$\mathit{abf}$ is also~$t$-inside~$abc$.
Hence, as~$c$ is~$t$-inside~$abc$ but not~$t^{*}$-inside~$\mathit{abf}$, 
by induction hypothesis, there exists
a common vertex to all cycles that~$t^{*}$-jump~$\mathit{abf}$. Also, by
(\ref{eq:every-cycle-in-cC-t'-jumps-abf}), there exists a common vertex
to all cycles in~$\cC$.
Moreover, this vertex is~$t$-inside~$abc$.

This concludes the proof of the lemma.
\end{proof}

\begin{corollary}\label{corollary:lct-tw3-lct>1-implies-not-jump}
Let~$\Delta$ be a triple of vertices in~$V_t$.
If~$\lct(G){>1}$ and for every pair of vertices~$\{i,j\}$ in~$\Delta$ there exists a longest cycle that 2-jumps
$\Delta~$ at~$\{i,j\}$, 
then there exists a longest cycle~$C$ in~$G$ such that either $C$ intersects $V_t$ at most once, or
one of the following possibilities is true:
\begin{itemize}
\item[$(i)$]~$C$ is~$t$-outside~$\Delta$,
\item[$(ii)$]~$C$ is~$t$-inside~$\Delta$ and {2-intersects}~$\Delta$,
\item[$(iii)$]~$C$ is~$t$-inside~$\Delta$, {3-intersects}~$\Delta$, and is fenced by~$\Delta$.
\end{itemize}
\end{corollary}
\begin{proof}
Without loss of generality, let~$\Delta=abc$.
For every~${ij \in \{ab,bc,ac\}}$,
let~$\cC_{ij}$ be the collection of all longest cycles in~$G$ that 2-jump
$abc$ at~$\{i,j\}$.
Let~$\cC_{abc}$ be the collection of all longest cycles in~$G$ that 3-jump~$abc$.
Let~${\cC = {\cC_{ab} \cup \cC_{bc} \cup \cC_{ac} \cup \cC_{abc}}}$.
As~${\cC_{ij}\neq \emptyset}$ for every~${ij \in \{ab,bc,ac\}}$,
by Lemma~\ref{lemma:lct-tw3-invariant}, all longest cycles 
in~$\cC$ intersect
at a common vertex, say~$x$, that is~$t$-inside~$abc$.
As~${\lct(G)>1}$, there exists a longest cycle~$C$ in~$G$ that is not in~$\cC$.
Hence,~$C$ does not jump~$abc$.
We may assume that $C$ intersects $V_t$ at least twice.
If~$C$ is~$t$-outside~$abc$, then~$(i)$ holds and we are done.
So, assume that~$C$ is~$t$-inside~$abc$.
If~$C$ {2-intersects}~$\Delta$, then~$(ii)$ holds and we are done.
Hence, we may assume that~$C$ {3-intersects}~$\Delta$.

Suppose by contradiction that~$C$ crosses~$\Delta$.
Let~$C_{ab},C_{bc}$, and~$C_{ac}$ be the~$abc$-parts of~$C$.
Let~$A$ be the component of~$G-\Delta$ where~$x$ lies.
As~$C$ crosses~$\Delta$, at least
one of~$\{C_{ab},C_{bc},C_{ac}\}$, say~$C_{ab}$, has all its internal
vertices not in~$A$.
Let~${D \in \cC_{ac}}$. Let~${F \in \cC_{bc}}$.
Let~$D'$ and~$D''$ be the two~$ac$-parts
of~$D$.
Let~$F'$ and~$F''$ be the two~$bc$-parts
of~$F$.
As~$D$ and~$F$ jump~$abc$, we may assume that~$D'$ and~$F'$ are~$t$-inside~$abc$,
and that~$D''$ and~$F''$ are~$t$-outside~$abc$.
Moreover, as~$D$ and~$F$ contain~$x$, and~$x$ is in~$A$, both~$D'$
and~$F'$ have all its internal vertices in~$A$.
As~$C_{ab}$ has all its internal
vertices not in~$A$,~$C_{ab}$ is internally disjoint from~$D'$ and~$F'$.
As~$C$ is~$t$-inside $abc$, $C_{ab}$ is internally disjoint from~$D''$ and~$F''$.
Hence,~${D' \cdot C_{ab} \cdot F''}$ and
${F' \cdot C_{ab} \cdot D''}$ are both cycles,
one of them longer than~$L$, a contradiction.
\end{proof}

%% file: Figures/lct-pairw-inters-1.tex
  \begin{tikzpicture}
[scale=0.7, label distance=3pt, every node/.style={fill,circle,inner sep=0pt,minimum size=6pt}]
 
    \node at (10,1.1547)[myblue,label=below:$a$,fill=black, circle](a) {};
    \node at (6,10.3923)[myblue,label=above left:$b$,fill=black, circle](b) {};
    \node at (0,0)[myblue,label=above left :$c$,fill=black, circle](c) {};
    \node at (12,0)[myblue,label=above right:$d$,fill=black, circle](d) {};
    
    \draw [dashed] (a) -- (b);
    \draw [dashed] (b) -- (c);
    \draw [dashed] (a) -- (c);
    \draw [dashed] (a) -- (d);
    \draw [dashed] (b) -- (d);
    \draw [dashed] (c) -- (d);
    
    \draw [color=red] [line width=1.5pt] (a) .. controls (5, 6) .. (b);
    
    \draw [color=mygreen] [line width=1pt] (a) .. controls (4, 3) .. (c);

    \draw [dashdotted][color=black] [line width=1pt] (b) .. controls (3.5, 5.5) and (8,7) .. (6,4) .. controls(5,2.3) and (8.5,0.3) .. (4,1.5) .. controls(2,2.4) .. (c) ;
    
    \node at (5.8,5.5)[myblue,label=below:$ $,fill=white, circle](X) {$\bar{C}'$};
    \node at (5,2.5)[myblue,label=below:$ $,fill=white, circle](X) {$\bar{D}'$};
    \node at (5.8,3.5)[myblue,label=below:$ $,fill=white, circle](X) {$F'$};
    \node at (6.4,4.8)[myblue,label=right:$x$,fill=black, circle](x) {};
    \node at (6.1,2.25)[myblue,label=above right:$y$,fill=black, circle](y) {};
  \end{tikzpicture}

%% file: Figures/lct-pairw-inters-2.tex
  \begin{tikzpicture}
[scale=0.7, label distance=3pt, every node/.style={fill,circle,inner sep=0pt,minimum size=6pt}]
 
    \node at (10,1.1547)[myblue,label=below:$a$,fill=black, circle](a) {};
    \node at (6,10.3923)[myblue,label=above left:$b$,fill=black, circle](b) {};
    \node at (0,0)[myblue,label=above left :$c$,fill=black, circle](c) {};
    \node at (12,0)[myblue,label=above right:$d$,fill=black, circle](d) {};
    
    \draw [dashed] (a) -- (b);
    \draw [dashed] (b) -- (c);
    \draw [dashed] (a) -- (c);
    \draw [dashed] (a) -- (d);
    \draw [dashed] (b) -- (d);
    \draw [dashed] (c) -- (d);
    
    \draw [color=red] [line width=1.5pt] (a) .. controls (5, 6) .. (b);
    
    \draw [color=mygreen] [line width=1pt] (a) .. controls (2, 2) .. (b);

    \draw [dashdotted][color=black] [line width=1pt] (b) .. controls (3.5, 5.5) and (8,7) .. (6,4) .. controls(5,1.3) and (8.5,-0.5) .. (4,2.6) .. controls(3,3.4) .. (c) ;
    
    \node at (5.8,5.5)[myblue,label=below:$ $,fill=white, circle](X) {$C'$};
    \node at (3.5,4.5)[myblue,label=below:$ $,fill=white, circle](X) {$D'$};
    \node at (5.8,3.5)[myblue,label=below:$ $,fill=white, circle](X) {$F'$};
    \node at (6.4,4.8)[myblue,label=right:$x$,fill=black, circle](x) {};
    \node at (6.1,1.68)[myblue,label=above right:$y$,fill=black, circle](y) {};
  \end{tikzpicture}

%% file: Figures/lct-pairw-inters-3.tex
  \begin{tikzpicture}
[scale=0.7, label distance=3pt, every node/.style={fill,circle,inner sep=0pt,minimum size=6pt}]
 
    \node at (10,1.1547)[myblue,label=below:$a$,fill=black, circle](a) {};
    \node at (6,10.3923)[myblue,label=above left:$b$,fill=black, circle](b) {};
    \node at (0,0)[myblue,label=above left :$c$,fill=black, circle](c) {};
    \node at (12,0)[myblue,label=above right:$d$,fill=black, circle](d) {};
    
    \draw [dashed] (a) -- (b);
    \draw [dashed] (b) -- (c);
    \draw [dashed] (a) -- (c);
    \draw [dashed] (a) -- (d);
    \draw [dashed] (b) -- (d);
    \draw [dashed] (c) -- (d);
    
    \draw [color=mygreen] [line width=1pt] (a) .. controls (5, 6) .. (b);
    
    \draw [color=red] [line width=1.5pt] (b) .. controls (3.6, 2) .. (c);

    \draw [dashdotted][color=black] [line width=1pt] (b) .. controls (3.5, 5.5) and (9.5,5.3) .. (5.5,5) .. controls(1,5) and (4.5,4.5) .. (4,2.5) .. controls (3,1.5) and (1.5,2) .. (c) ;
    
    \node at (7.8,3.4)[myblue,label=below:$ $,fill=white, circle](X) {$D'$};
    \node at (4.8,6)[myblue,label=below:$ $,fill=white, circle](X) {$C_{bc}$};
    \node at (5.5,5)[myblue,label=below:$ $,fill=white, circle](X) {$R$};
    \node at (6.25,5.05)[myblue,label=below:$x$,fill=black, circle](x) {};
    \node at (4.4,5)[myblue,label=below right:$y$,fill=black, circle](y) {};
  \end{tikzpicture}

%% file: Figures/lct-pairw-inters-4.tex
  \begin{tikzpicture}
[scale=0.7, label distance=3pt, every node/.style={fill,circle,inner sep=0pt,minimum size=6pt}]
 
    \node at (10,1.1547)[myblue,label=below:$a$,fill=black, circle](a) {};
    \node at (6,10.3923)[myblue,label=above left:$b$,fill=black, circle](b) {};
    \node at (0,0)[myblue,label=above left :$c$,fill=black, circle](c) {};
    \node at (12,0)[myblue,label=above right:$d$,fill=black, circle](d) {};
    
    \draw [dashed] (a) -- (b);
    \draw [dashed] (b) -- (c);
    \draw [dashed] (a) -- (c);
    \draw [dashed] (a) -- (d);
    \draw [dashed] (b) -- (d);
    \draw [dashed] (c) -- (d);
    
    \draw [color=red] [line width=1.5pt] (a) .. controls (5, 3) .. (b);
    \draw [color=mygreen] [line width=1pt] (a) .. controls (7, 5) .. (b);
    
    \draw [color=red] [line width=1.5pt] (b) .. controls (3, 2) .. (c);
    \draw [color=mygreen] [line width=1pt] (b) .. controls (4, 2) .. (c);

    \draw [dashdotted][color=black] [line width=1pt] (a) .. controls (1.5, 1.5) and (11.5,3) .. (7.25,5.05) .. controls(1,5.5) and (4.5,4.5) .. (4,2.5) .. controls (3,1.5) and (1.5,2) .. (c) ;
    
    \node at (6.8,7)[myblue,label=below:$ $,fill=white, circle](X) {$D_{ab}$};
    \node at (4.2,6)[myblue,label=below:$ $,fill=white, circle](X) {$C_{bc}$};
    \node at (6.2,3)[myblue,label=below:$ $,fill=white, circle](X) {$C_{ab}$};
     \node at (6.3,5.2)[myblue,label=below:$ $,fill=white, circle](X) {$R$};
    \node at (4.5,4.2)[myblue,label=below:$ $,fill=white, circle](X) {$D_{bc}$};
    \node at (7.28,5.05)[myblue,label=below:$y$,fill=black, circle](y) {};
    \node at (5.5,5.15)[myblue,label=below left:$x$,fill=black, circle](x) {};
  \end{tikzpicture}

%% file: Figures/lct-pairw-inters-5.tex
  \begin{tikzpicture}
[scale=0.7, label distance=3pt, every node/.style={fill,circle,inner sep=0pt,minimum size=6pt}]
 
    \node at (10,1.1547)[myblue,label=below:$a$,fill=black, circle](a) {};
    \node at (6,10.3923)[myblue,label=above left:$b$,fill=black, circle](b) {};
    \node at (0,0)[myblue,label=above left :$c$,fill=black, circle](c) {};
    \node at (12,0)[myblue,label=above right:$d$,fill=black, circle](d) {};
    
    \draw [dashed] (a) -- (b);
    \draw [dashed] (b) -- (c);
    \draw [dashed] (a) -- (c);
    \draw [dashed] (a) -- (d);
    \draw [dashed] (b) -- (d);
    \draw [dashed] (c) -- (d);
    
    \draw [color=red] [line width=1.5pt] (a) .. controls (5, 3) .. (b);
    \draw [color=mygreen] [line width=1pt] (a) .. controls (7, 5) .. (b);
    
    \draw [color=red] [line width=1.5pt] (b) .. controls (3, 2) .. (c);
    \draw [color=mygreen] [line width=1pt] (a) .. controls (5, 2.5) .. (c);

    \draw [dashdotted][color=black] [line width=1pt] (a) .. controls (1.5, 1.5) and (11.5,3) .. (7.25,5.05) .. controls(1,5.5) and (11,6.5) .. (5.5,7.5) .. controls (3.5,8) and (6.5,8) .. (b) ;
    
    \node at (7,5.7)[myblue,label=below:$ $,fill=white, circle](X) {$D_{ab}$};
    \node at (4.2,6)[myblue,label=below:$ $,fill=white, circle](X) {$C_{bc}$};
    \node at (5.7,4)[myblue,label=below:$ $,fill=white, circle](X) {$C_{ab}$};
     \node at (7.6,3)[myblue,label=below:$ $,fill=white, circle](X) {$R$};
    \node at (4.5,2)[myblue,label=below:$ $,fill=white, circle](X) {$D_{ac}$};
    \node at (8.24,3.5)[myblue,label=below:$y$,fill=black, circle](y) {};
    \node at (7,2.4)[myblue,label=above:$x$,fill=black, circle](x) {};
  \end{tikzpicture}

%% file: Figures/lct-pairw-inters-6.tex
  \begin{tikzpicture}
[scale=0.7, label distance=3pt, every node/.style={fill,circle,inner sep=0pt,minimum size=6pt}]
 
    \node at (10,1.1547)[myblue,label=below:$a$,fill=black, circle](a) {};
    \node at (6,10.3923)[myblue,label=above left:$b$,fill=black, circle](b) {};
    \node at (0,0)[myblue,label=above left :$c$,fill=black, circle](c) {};
    \node at (12,0)[myblue,label=above right:$d$,fill=black, circle](d) {};
    
    \draw [dashed] (a) -- (b);
    \draw [dashed] (b) -- (c);
    \draw [dashed] (a) -- (c);
    \draw [dashed] (a) -- (d);
    \draw [dashed] (b) -- (d);
    \draw [dashed] (c) -- (d);
    
    \draw [color=red] [line width=1.5pt] (a) .. controls (5, 3) .. (b);
    \draw [color=mygreen] [line width=1pt] (a) .. controls (7, 5) .. (b);
    
    \draw [color=red] [line width=1.5pt] (b) .. controls (3, 2) .. (c);
    \draw [color=mygreen] [line width=1pt] (a) .. controls (5, 2.5) .. (c);

    \draw [dashdotted][color=black] [line width=1pt] (a) .. controls (7.5, 4.5) and (6.5,0.5) .. (5,2.05) .. controls(-1,5.5) and (11,6.5) .. (5.5,7.5) .. controls (3.5,8) and (6.5,8) .. (b) ;
    
    \node at (7.2,5.3)[myblue,label=below:$ $,fill=white, circle](X) {$D_{ab}$};
    \node at (4.2,6)[myblue,label=below:$ $,fill=white, circle](X) {$C_{bc}$};
    \node at (5.7,4)[myblue,label=below:$ $,fill=white, circle](X) {$C_{ab}$};
     \node at (4,2.7)[myblue,label=below:$ $,fill=white, circle](X) {$R$};
    \node at (3.8,1.7)[myblue,label=below:$ $,fill=white, circle](X) {$D_{ac}$};
    \node at (5,2)[myblue,label=below:$y$,fill=black, circle](y) {};
    \node at (3.35,3.5)[myblue,label=left:$x$,fill=black, circle](x) {};
  \end{tikzpicture}

%% file: sectionConcludingRemarks.tex
In this paper, we showed
all longest cycles in a 2-connected partial $3$-tree intersect.
This is not true for partial 4-trees.
Indeed, there exists a 2-connected partial 4-tree~$G$
given by Thomassen on
15 vertices~\cite[Figure 16]{Shabbir13} in which not all longest cycles intersect, but there is a pair of vertices meeting all longest cycles
Hence, we propose the following conjecture.
\begin{conjecture}
For every 2-connected partial 4-tree, there exists a set of two vertices such that all longest cycles has at least one vertex in this set.
\end{conjecture}

